\documentclass[11pt]{amsart}
\usepackage[top=1.5in, bottom=1.5in, left=1.4in, right=1.4in]{geometry} 
\usepackage{amssymb}
\usepackage{amsmath}
\usepackage{mathrsfs}
\input xy
\xyoption{all}
\usepackage{mathdots}
\usepackage{mathtools}

\usepackage{url}
\usepackage{enumerate}
\usepackage{csquotes}
\usepackage{array}

\usepackage[all]{xy}
\usepackage{graphicx}  
\usepackage{tikz}
\usetikzlibrary{decorations.pathreplacing}
\usepackage{tkz-graph}
\usetikzlibrary{arrows}
\usetikzlibrary{decorations.markings}

\usepackage{algorithm}
\usepackage[noend]{algpseudocode} 

\algnewcommand{\algorithmicgoto}{\textbf{go to}}%
\algnewcommand{\Goto}[1]{\algorithmicgoto~line~\ref{#1}}%

\usepackage{hyperref}

\newtheorem{thm}{Theorem}[section]

\newtheorem{cor}[thm]{Corollary}

\newtheorem{fact}[thm]{Fact}

\newtheorem{Definition}[thm]{Definition}
\newenvironment{definition}
  {\begin{Definition}\rm}{\end{Definition}}
\newtheorem{Example}[thm]{Example}
\newenvironment{example}
  {\begin{Example} \rm}{\end{Example}}
\newtheorem{Remark}[thm]{Remark}
\newenvironment{remark}
  {\begin{Remark}\rm}{\end{Remark}}
  
\numberwithin{equation}{section}

\usepackage{etoolbox}
\apptocmd{\sloppy}{\hbadness 10000\relax}{}{}

\def\SetFancyGraph {
	\SetVertexMath
	\GraphInit[vstyle=Art]
	\SetUpVertex[MinSize=2pt]
	\SetVertexLabel
	\tikzset{VertexStyle/.style = {shape = circle,shading = ball,ball color = black,inner sep = 1.5pt}}
	\SetUpEdge[color=black]
	\tikzstyle{EdgeStyle}=[auto]
	\tikzstyle{LabelStyle}=[auto=3pt,inner sep=0pt,outer sep=3pt]
	\tikzset{->-/.style={decoration={ markings, mark=at position 0.8 with {\arrow{>}}},postaction={decorate}}}
	\tikzset{->--/.style={decoration={ markings, mark=at position 0.55 with {\arrow{>}}},postaction={decorate}}}
	\tikzstyle{chips}  = [circle, minimum width=15pt, fill=white,draw=black, inner sep=0pt]
}

\title{Parking functions and tree inversions revisited}
\author{Petar Gaydarov}
\email{peter.gaydaro@gmail.com}
\author{Sam Hopkins}
\email{shopkins@mit.edu}
\address{Massachusetts Institute of Technology, Cambridge, MA, 02139, United States}

\begin{document}

\begin{abstract}
Kreweras proved that the reversed sum enumerator for parking functions of length $n$ is equal to the inversion enumerator for labeled trees on $n+1$ vertices. Recently, Perkinson, Yang, and Yu gave a bijective proof of this equality that moreover generalizes to graphical parking functions. Using a depth-first search variant of Dhar's burning algorithm they proved that the reversed sum enumerator for $G$-parking functions equals the $\kappa$-number enumerator for spanning trees of $G$. The $\kappa$-number is a kind of generalized tree inversion number originally defined by Gessel. We extend the work of Perkinson-Yang-Yu to what are referred to as ``generalized parking functions'' in the literature, but which we prefer to call \emph{vector parking functions} because they depend on a choice of vector $\mathbf{x} \in \mathbb{N}^n$. Specifically, we give an expression for the reversed sum enumerator for $\mathbf{x}$-parking functions in terms of inversions in rooted plane trees with respect to certain admissible vertex orders. Along the way we clarify the relationship between graphical and vector parking functions.
\end{abstract}

\maketitle

\section{Introduction} \label{sec:intro}

Set $\mathbb{N} := \{0,1,2,\ldots\}$. Parking functions are certain sequences in $\mathbb{N}^n$ that originally arose in the study of hashing functions~\cite{konheim1966occupancy}. Their name comes from the following description in terms of parking preferences for cars:
\begin{displayquote}
Imagine that $n$ parking spots labeled $0$ through $n-1$ are arranged in order on a linear street. Cars $C_1,\ldots,C_n$ approach the spaces in order. Car $C_i$ prefers spot $\alpha_i$, which means that he will drive until he reaches that spot and park there it if it is unoccupied. If his preferred spot is occupied he will continue driving until he comes to an unoccupied spot and park in this unoccupied spot. If his preferred spot and all the spots after it are occupied, then he cannot park. A parking function of length~$n$ is sequence $(\alpha_1,\ldots,\alpha_n) \in \mathbb{N}^n$ of parking preferences so that all the cars can park.
\end{displayquote}
An equivalent, less whimsical, definition of parking function is the following.
\begin{definition}
A \emph{parking function of length $n$} is a sequence $(\alpha_1,\ldots,\alpha_n) \in \mathbb{N}^{n}$ whose weakly increasing rearrangement $\alpha_{i_1} \leq \alpha_{i_2} \leq \cdots \leq \alpha_{i_n}$ satisfies $\alpha_{i_j} \leq j -1$ for all~$1 \leq j \leq n$. We denote the set of parking functions of length $n$ by $\mathrm{PF}(n)$.
\end{definition}

Parking functions have appeared in various of areas of pure mathematics. We will not attempt any kind of survey here except to say that, beyond their inherent combinatorial interest, parking functions arise in the study of hyperplane arrangements~\cite[Lecture~6]{stanley2007introduction} and the theory of diagonal harmonics~\cite[Chapter~5]{haglund2008qtcatalan}. For more background on parking functions in general, see the recent survey by Yan~\cite{yan2015parking}.

Parking functions are intimately related to labeled trees. For instance, it is well-known that the number of parking functions of length~$n$ is~$(n+1)^{n-1}$~\cite{konheim1966occupancy, riordan1969ballots}. Cayley's formula tells us that this is also the number of labeled trees on $n+1$ vertices. Indeed, there are many explicit bijections between parking functions and labeled trees~\cite{schutzenberger1968enumeration, foata1974mappings}. But more is true. The \emph{sum} of a parking function~$\alpha = (\alpha_1,\ldots,\alpha_n) \in \mathrm{PF}(n)$ is defined to be~$\mathrm{sum}(\alpha) := \sum_{i=1}^{n} \alpha_i$ and the \emph{reversed sum} is~$\mathrm{rsum}(\alpha) := \binom{n}{2} - \mathrm{sum}(\alpha)$. The reversed sum is defined so that the parking functions of maximal sum have reversed sum~$0$. The reversed sum is also called the ``area of a major sequence'' in~\cite{beissinger1997note} and ``total displacement'' in~\cite{knuth1998linear} but we follow the terminology used in~\cite{kung2003goncarov}. Let us call the expression $\sum_{\alpha\in \mathrm{PF}(n)} q^{\mathrm{rsum}(\alpha)}$ the \emph{reversed sum enumerator} for parking functions of length~$n$. Kreweras~\cite{kreweras1980famille} established that the reversed sum enumerator for parking functions of length~$n$ is equal to the \emph{inversion enumerator} for labeled trees on~$n+1$ vertices. That is, Kreweras proved that
\begin{align} 
\sum_{\alpha \in \mathrm{PF}(n)} q^{\mathrm{rsum}(\alpha)} = \sum_{T} q^{\mathrm{inv}(T)} \label{eqn:classical}
\end{align}
where the sum over all labeled trees $T$ on $n+1$ vertices and $\mathrm{inv}(T)$ is the \emph{inversion number} of  $T$, a certain natural statistic that generalizes the inversion number of a permutation. Specifically, for $T$ a labeled tree with vertices $0,1,2,\ldots,n$ (which we will always consider rooted at~$0$) an \emph{inversion of $T$} is a pair $(i,j)$ with $1 \leq i < j \leq n$ such that the unique path from $i$ to $0$ passes through $j$. The \emph{inversion number}~$\mathrm{inv}(T)$ is the number of inversions of $T$. 

\begin{example}
Set $n := 2$. The three labeled trees on $3$ vertices and their inversion numbers are the following:
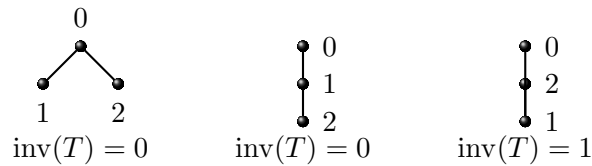
\begin{center}
\begin{tikzpicture}[scale=0.5]
	\SetFancyGraph
	\Vertex[LabelOut,Lpos=90, Ldist=.05cm,x=0,y=1]{0}
	\Vertex[LabelOut,Lpos=270, Ldist=.05cm,x=-1,y=0]{1}
	\Vertex[LabelOut,Lpos=270, Ldist=.05cm,x=1,y=0]{2}
	\Edges[style={thick}](0,1)
	\Edges[style={thick}](0,2)
	\node at (0,-1.75) {$\mathrm{inv}(T) = 0$};
\end{tikzpicture} \qquad \begin{tikzpicture}[scale=0.5]
	\SetFancyGraph
	\Vertex[LabelOut,Lpos=0, Ldist=.05cm,x=0,y=1]{0}
	\Vertex[LabelOut,Lpos=0, Ldist=.05cm,x=0,y=0]{1}
	\Vertex[LabelOut,Lpos=0, Ldist=.05cm,x=0,y=-1]{2}
	\Edges[style={thick}](0,1)
	\Edges[style={thick}](1,2)
	\node at (0,-1.75) {$\mathrm{inv}(T) = 0$};
\end{tikzpicture} \qquad \begin{tikzpicture}[scale=0.5]
	\SetFancyGraph
	\Vertex[LabelOut,Lpos=0, Ldist=.05cm,x=0,y=1]{0}
	\Vertex[LabelOut,Lpos=0, Ldist=.05cm,x=0,y=0]{2}
	\Vertex[LabelOut,Lpos=0, Ldist=.05cm,x=0,y=-1]{1}
	\Edges[style={thick}](0,1)
	\Edges[style={thick}](1,2)
	\node at (0,-1.75) {$\mathrm{inv}(T) = 1$};
\end{tikzpicture}
\end{center}
On the other hand, the three parking functions of length $2$ are $(0,0)$, $(1,0)$, and $(0,1)$, which have reversed sums of $1$, $0$, and $0$ respectively. So~(\ref{eqn:classical}) holds in this case.
\end{example}

For more on the inversion enumerator for labeled trees see~\cite{mallows1968inversion} and~\cite{gessel1995enumeration}. Kreweras's proof of~(\ref{eqn:classical}) was via a recursively defined bijection between parking functions and trees that preserves the requisite statistics. Stanley stated the problem of finding a nonrecursive bijective proof of~(\ref{eqn:classical}) as~\cite[Lecture 6, Exercise 4]{stanley2007introduction}. Various authors~\cite{shin2008new, guedes2010parking} devised such a nonrecursive bijection. Recently, Perkinson, Yang, and Yu~\cite{perkinson2013gparking} found an elegant bijective proof of~(\ref{eqn:classical}) that moreover generalizes naturally to the setting of \emph{graphical parking functions}. To keep them distinct from their generalizations, let us refer to the regular parking functions defined above as \emph{classical parking functions}.

Graphical parking functions were first given that name by Postnikov and Shaprio~\cite{postnikov2004trees} but have been studied for longer in the context of the abelian sandpile model where they are essentially the same as what are called \emph{superstable configurations} (see~\cite[Definition~4.3]{holroyd2008chip} or~\cite[\S2.3]{perkinson2013primer}). The superstable configurations are dual to the \emph{recurrent configurations} (see~\cite[Definition 2.11]{holroyd2008chip} or~\cite[Definition~2.10]{perkinson2013primer}; these recurrent configurations are also sometimes called \emph{critical configurations} as in~\cite{biggs1999chip}). The set of recurrent configurations is a more natural object to study from the perspective of statistical mechanics, the field out of which the abelian sandpile model originally arose. Indeed, the abelian sandpile model was first described in the pioneering work of Bak, Tang, and Wiesenfeld~\cite{bak1987self}, who were interested in modeling self-organized criticality in physical systems. Dhar~\cite{dhar1990self} subsequently put the model into a more general graphical framework and developed much of the basic theory. Let us also note that graphical parking functions appear under the name of \emph{$v_0$-reduced divisors} in Riemann-Roch theory for graphs~\cite{baker2007riemann}. As such they are related to a deep analogy between finite graphs and Riemann surfaces. For more background on sandpile theory, consult the short survey~\cite{levine2010sandpile} or the upcoming book~\cite{corry2016divisors}.

Graphical parking functions depend on a choice of graph $G$ (as well as a choice of root vertex that we will ignore in this introduction). We denote the set of $G$-parking functions by~$\mathrm{PF}(G)$. For~$\alpha \in \mathrm{PF}(G)$ we have similar notions of \emph{sum}~$\mathrm{sum}(\alpha)$ (which in this context is often called \emph{degree}) and \emph{reversed sum}~$\mathrm{rsum}(\alpha)$.  $G$-parking functions serve as coset representatives for the cokernel of the reduced Laplacian~$\widetilde{\Delta}(G)$ of the graph $G$. Therefore, by Kirchhoff's Matrix-Tree Theorem, the number of $G$-parking functions is equal to the number of spanning trees of $G$. It is thus natural to ask for an extension of Kreweras's result~(\ref{eqn:classical}) to this graphical setting which says that the reversed sum enumerator for $G$-parking functions is some kind of ``inversion enumerator'' for spanning trees of~$G$.

Perkinson-Yang-Yu~\cite{perkinson2013gparking} provided such an extension. Specifically, they gave a bijective proof that
\begin{equation} \label{eqn:graphical}
 \sum_{\alpha \in \mathrm{PF}(G)} q^{\mathrm{rsum}(\alpha)} = \sum_{T} q^{\kappa(G,T)}
\end{equation}
for any simple graph $G$, where the sum is over all spanning trees~$T$ of~$G$ and~$\kappa(G,T)$ is a kind of generalized inversion number originally defined by Gessel~\cite{gessel1995enumerative}. Importantly, when $G=K_{n+1}$ is the complete graph, equation~(\ref{eqn:graphical}) recovers~(\ref{eqn:classical}). The proof of~(\ref{eqn:graphical}) in~\cite{perkinson2013gparking} is based on a depth-first search (DFS) variant of Dhar's burning algorithm~\cite{dhar1990self, dhar2006theoretical}. The idea of using depth-first search to study tree inversions goes back at least to the fundamental work of Gessel and Wang~\cite{gessel1979depth}. Depth-first search appeared again in the paper~\cite{gessel1995enumerative} in which Gessel defined the $\kappa$-number statistic. And depth-first search, as well as the related \emph{neighbor-first search}, were also used later by Gessel and Sagan~\cite{gessel1996tutte} to relate certain spanning tree activities (including generalized inversion numbers) to the Tutte polynomial of a graph and to classical parking functions. The main novelty of~\cite{perkinson2013gparking} was the combination of depth-first search with Dhar's burning algorithm. (But note that similar ideas appeared in~\cite[\S3]{chebikin2005family} and especially in~\cite[\S4]{kostic2008multiparking}, which however used neighbor-first search rather than depth-first search per se.) Dhar's burning algorithm is one of the most important tools in the study of the abelian sandpile model: it is a linear time algorithm for deciding whether a given sequence is a $G$-parking function that in some sense generalizes the ``parking procedure'' mentioned at the beginning of this section. Another variant of the burning algorithm, due to Cori and Le~Borgne~\cite{cori2003sand}, proves that the reversed sum enumerator for~$G$-parking functions is also the generating function for spanning trees of $G$ by \emph{external activity} (a different tree statistic that depends on a total edge order rather than a total vertex order).

\begin{remark}
There is a natural embedding of the symmetric group $\mathfrak{S}_n$ on $n$ letters into the set of labeled trees on $n+1$ vertices given by $\sigma \mapsto T_{\sigma}$ where $T_{\sigma}$ is
\begin{center}
\begin{tikzpicture}[scale=1]
	\SetFancyGraph
	\Vertex[LabelOut,Lpos=90, Ldist=.1cm,x=0,y=0]{0}
	\Vertex[LabelOut,Lpos=90, Ldist=.1cm,x=1,y=0,L={\sigma(1)}]{1}
	\Vertex[LabelOut,Lpos=90, Ldist=.1cm,x=2,y=0,L={\sigma(2)}]{2}
	\Vertex[LabelOut,Lpos=90, Ldist=.1cm,x=3,y=0,L={\sigma(n)}]{3}	
	\Edges[style={thick}](0,1)
	\Edges[style={thick}](1,2)
	\Edges[style={thick,dashed}](2,3)
\end{tikzpicture}
\end{center}
Observe that $\mathrm{inv}(T_{\sigma}) = \mathrm{inv}(\sigma)$, the usual permutation inversion number. Pushing $T_{\sigma}$ through the bijection of Perkinson-Yang-Yu~\cite{perkinson2013gparking} yields not the most obvious way of representing $\sigma$ in $\mathrm{PF}(n)$ (``one-line notation'') but rather a version of the \emph{code} of $\sigma$ (see~\cite[\S1.3]{stanley2012ec1}). Indeed, the identity~(\ref{eqn:classical}) should be seen as a generalization of the very classical result
\[ [n]_q! = \sum_{\sigma \in \mathfrak{S}_n} q^{\mathrm{inv}(\sigma)}\]
where $[n]_q! := [n]_q \cdot [n-1]_q \cdots [2]_q \cdot [1]_q$ and $[k]_q := \frac{1-q^{k}}{1-q} = 1 + q + q^2 + \cdots + q^{k-1}$ are the standard \emph{$q$-factorials} and \emph{$q$-numbers}. This offers some impression of the subtlety of the DFS-burning bijection.
\end{remark}

There is another generalization of classical parking function that goes under the unfortunately vague name of \emph{generalized parking function}~\cite{stanley2002polytope}~\cite{yan2000enumeration}. Whereas graphical parking functions depend on a choice of graph $G$, generalized parking functions depend on a choice of vector $\mathbf{x} = (x_1,\ldots,x_n) \in \mathbb{N}^n$ of nonnegative integers. As such we will call them \emph{$\mathbf{x}$-parking functions} or \emph{vector parking functions}.

\begin{definition} \label{def:xpf}
Let $\mathbf{x} = (x_1,\ldots,x_n) \in \mathbb{N}^n$ be some vector of nonnegative integers. An~\emph{$\mathbf{x}$-parking function} is a sequence $(\alpha_1,\ldots,\alpha_n) \in \mathbb{N}^{n}$ whose weakly increasing rearrangement $\alpha_{i_1} \leq \alpha_{i_2} \leq \cdots \leq \alpha_{i_n}$ satisfies $\alpha_{i_j} \leq (\sum_{k=1}^{j} x_k) -1$ for all~$1 \leq j \leq n$. We denote the set of $\mathbf{x}$-parking functions by $\mathrm{PF}(\mathbf{x})$.
\end{definition}

Observe that $\mathrm{PF}(n) = \mathrm{PF}(\mathbf{x})$ for $\mathbf{x} = (\overbrace{1,1,\ldots,1}^{n})$. Also note that \emph{rational parking functions}, which because of their connections to algebra and geometry (especially the representation theory of rational Cherednik algebras~\cite{berest2003finite}) have been much studied as of late~\cite{gorsky2014affine, armstrong2014rational}, are a special case of vector parking functions. But vector parking functions are less structured objects than rational parking functions: for instance, it is easy to count rational parking functions (see~\cite[Corollary~4]{armstrong2014rational}) but apparently not so easy to count arbitrary vector parking functions. The best general formula for enumerating $\mathbf{x}$-parking functions is due to Pitman and Stanley~\cite[Theorems 1 and~11]{stanley2002polytope} and appeared in the paper in which they introduced $\mathbf{x}$-parking functions. In order to state their result we need a little more notation: for all $n \in \mathbb{N}$, set 
\[\Gamma(n) := \{(\gamma_1,\ldots,\gamma_{n}) \in \mathbb{N}^n\colon \sum_{i=1}^{j}\gamma_i \geq j \textrm{ for all $1\leq i \leq n-1$ and } \sum_{i=1}^{n}\gamma_i = n\}.\]
It is well-known that $\#\Gamma(n) = C_n$ where $C_n := \frac{1}{n+1}\binom{2n}{n}$ is the $n$th Catalan number; see e.g.~\cite[\S2~Exercise~86]{stanley2015catalan}.

\begin{thm}[Pitman-Stanley] \label{thm:pitmanstanley} For any $\mathbf{x} = (x_1,\ldots,x_{n}) \in \mathbb{N}^n$, the number of $\mathbf{x}$-parking functions is the following polynomial $P_n(x_1,\ldots,x_{n})$ in the variables $x_1,\ldots,x_{n}$:
\begin{align*}
P_n(x_1,\ldots,x_{n}) &:= \sum_{(\alpha_1,\ldots,\alpha_n) \in \mathrm{PF}(n)} x_{\alpha_1+1} x_{\alpha_2+1} \cdots x_{\alpha_n+1} \\
&= \sum_{(\gamma_1,\ldots,\gamma_n) \in \Gamma(n)} \frac{n!}{\gamma_1!\gamma_2!\cdots\gamma_n!} x_1^{\gamma_1}x_2^{\gamma_2}\cdots x_{n}^{\gamma_n}.
\end{align*}
\end{thm}

Vector parking functions have been studied in many subsequent papers, especially by Yan and her collaborators. In~\cite{yan2000enumeration} Yan gave nice enumerative formulas for several special classes of $\mathbf{x}$-parking functions. In~\cite{yan2001generalized}, building on earlier work~\cite{yan1997generalized}, Yan related the reversed sum enumerator for a special class of $\mathbf{x}$-parking functions to tree inversions. In~\cite{kung2003goncarov} and~\cite{kung2003expected} Kung and Yan studied the connection between $\mathbf{x}$-parking functions and Gon\v{c}arov polynomials. Via this connection they offered linear recursions, a shift formula, and a perturbation formula for the number of $\mathbf{x}$-parking functions, as well as a formula for the expected sum of a random $\mathbf{x}$-parking function. They also investigated the reversed sum enumerator for $\mathrm{PF}(\mathbf{x})$. For~$\mathbf{x} = (x_1,\ldots,x_n) \in \mathbb{N}^n$, define the \emph{sum} of a vector parking function $\alpha = (\alpha_1,\ldots,\alpha_n) \in \mathrm{PF}(\mathbf{x})$ to be $\mathrm{sum}(\alpha) := \sum_{i=1}^{n} \alpha_i$ and the \emph{reversed sum} to be~$\mathrm{rsum}(\alpha) :=  \left( \sum_{i=1}^{n} (n+1-i) x_{i} \right) - n - \mathrm{sum}(\alpha)$. Again, reversed sum is defined so that $\mathbf{x}$-parking functions of maximal sum have reversed sum~$0$. Kung and Yan~\cite[Theorem~6.2]{kung2003goncarov} established the following formula for the reversed sum enumerator for~$\mathrm{PF}(\mathbf{x})$ (and they also explained how it follows from the arguments of Pitman-Stanley~\cite{stanley2002polytope}).

\begin{thm}[Pitman-Stanley and Kung-Yan] \label{thm:kungyan}
For any $\mathbf{x} = (x_1,\ldots,x_n) \in \mathbb{N}^n$, the reversed sum enumerator for $\mathbf{x}$-parking functions is 
\begin{align*}
\sum_{\alpha \in \mathrm{PF}(\mathbf{x})} q^{\mathrm{rsum}(\alpha)} &= q^{ \left( \sum_{i=1}^{n} (n+1-i) x_{i} \right) - n} \, P_n([x_1]_{1/q},q^{-x_1}[x_2]_{1/q},\ldots,q^{-x_1-x_2\cdots-x_{n-1}}[x_n]_{1/q}) \\
&= \sum_{(\gamma_1,\ldots,\gamma_n) \in \Gamma(n)}  \frac{n!}{\gamma_1!\gamma_2!\cdots\gamma_n!} q^{\sum_{i=1}^{n-1}(\gamma_{1}+\gamma_{2} + \cdots+\gamma_i - i)x_{i+1}}\left( \prod_{i=1}^{n}[x_i]_q^{\gamma_i} \right).
\end{align*}
\end{thm}

Since the best general enumerative formula for~$\#\mathrm{PF}(\mathbf{x})$ (Theorem~\ref{thm:pitmanstanley}) involves a sum of Catalan many terms we cannot hope for a better formula for the reversed sum enumerator for~$\mathrm{PF}(\mathbf{x})$ than one which also sums Catalan many terms. Thus in some sense Theorem~\ref{thm:kungyan} is totally satisfactory. But on the other hand, in light of the connections between tree inversions and the reversed sum enumerators for classical and graphical parking functions, we might hope that we could also express the reversed sum enumerator for vector parking functions in terms of tree inversions. In particular we could ask for an expression for the reversed sum enumerator for $\mathbf{x}$-parking functions that evidently recaptures~(\ref{eqn:classical}) when $\mathbf{x} := (1,1,\ldots,1)$ and Theorem~\ref{thm:pitmanstanley} when $q:=1$. This is the problem we take up in the current paper.

We adapt the technique of Perkinson-Yang-Yu to give an expression for the reversed sum enumerator for $\mathbf{x}$-parking functions in terms of tree inversions. In the special case where~$\mathbf{x} = (a,b,b,\ldots,b) \in \mathbb{N}^n$ we recover a result of Yan~\cite{yan2001generalized} that gives the reversed sum enumerator for $\mathrm{PF}(\mathbf{x})$ in terms of inversions in spanning trees of a certain multigraph~$K^{a,b}_{n+1}$ (see Remark~\ref{rem:yan}). In this case the~$\mathbf{x}$-parking functions are exactly the graphical parking functions of~$K^{a,b}_{n+1}$. In fact, the first step of our proof is to extend in a straightforward way the DFS-burning algorithm of~\cite{perkinson2013gparking} to multigraphs (see Theorem~\ref{thm:gpfrsum}).  But our result is more general than that of~\cite{yan2001generalized} because it applies to arbitrary~$\mathbf{x}$. In general vector parking functions are not graphical parking functions of any multigraph and vice-versa. We completely characterize the (small) overlap between graphical and vector parking functions in Theorem~\ref{thm:gpfandxpf}.

Our main result expresses the reversed sum enumerator for~$\mathrm{PF}(\mathbf{x})$ as a sum over rooted plane trees. In order to state this result, let us briefly review notation for trees. A \emph{rooted tree} is a tree together with a choice of distinguished \emph{root} vertex. A rooted tree~$T$ comes with a partial order~$\leq_{T}$ on its vertex set~$V(T)$ whereby~$i \leq_T j$ if and only if the unique path in~$T$ from~$j$ to the root passes through~$i$.  If~$i \leq_T j$ and~$i \neq j$ then we say  $i$ is an \emph{ancestor} of~$j$ and $j$ is a \emph{descendant} of $i$. We use other familial language in the obvious way: so we say that~$j$ is a \emph{child} of~$i$, and $i$ is the \emph{parent} of~$j$, if~$j$ is a descendant of $i$ and no descendant of~$i$ is an ancestor of~$j$. For a non-root vertex~$i$, we use~$\mathrm{par}_T(j)$ to denote the (unique) parent of~$j$. We say that $i$ and $j$ are \emph{siblings} if~$\mathrm{par}_T(i) = \mathrm{par}_T(j)$. A \emph{rooted plane tree} is an rooted tree drawn in the plane with parents drawn above children and with the children of each vertex linearly ordered from left to right. One fundamental combinatorial interpretation of the Catalan number $C_n$ is as the number of rooted plane trees on~$n+1$ vertices (see~\cite[Theorem 1.5.1(iii)]{stanley2015catalan}). By giving the tree a depth-first labeling (or \emph{preorder} as in~\cite[p.~10]{stanley2015catalan}) of its vertices, we have the following equivalent definition of rooted plane tree.

\begin{definition} \label{def:planetree}
For $n \geq 1$, a \emph{rooted plane tree} on $n+1$ vertices is a rooted tree~$T$ with vertex set~$V(T) := \{0,1,\ldots,n\}$ and root~$0$ satisfying the following conditions:
\begin{itemize}
\item $i \leq_T j$ implies $i \leq j$ for all $1 \leq i,j \leq n$;
\item $i \leq_T k$ implies $i \leq_T j$ for all $1 \leq i < j < k \leq n$. 
\end{itemize}
We denote the set of rooted plane trees on~$n+1$ vertices by $\mathrm{RPT}(n+1)$. 
\end{definition}

To state our main result we need also to discuss vertex orders for rooted trees. For~$T$ a rooted tree and $\prec$ any total order on its non-root vertices, an \emph{inversion of $T$ with respect to~$\prec$} is a pair~$(i,j)$ of non-root vertices with $i \leq_T j$ but $j \prec i$.  For~$T \in \mathrm{RPT}(n+1)$, an \emph{admissible vertex order} of $T$ is a total order~$\prec$ on~$V(T) - \{0\}$ for which~$1\leq i < j \leq n$ with~$i$ and~$j$ siblings implies~$j \prec i$. We denote the set of admissible vertex orders of $T$ by~$\mathrm{AVO}(T)$.  Observe that the number of pairs $(T,\prec)$ with~$T \in \mathrm{RPT}(n+1)$ and~$\prec \in \mathrm{AVO}(T)$ is~$(n+1)^{n-1}$, the number of labeled trees on~$n+1$ vertices. Indeed, given any labeled tree $T'$ with vertices~$0,1,\ldots,n$ there is a canonical way to associate such a pair $(T,\prec)$: perform a depth-first search of $T'$ starting at~$0$ and always preferring to visit the vertex with the greatest label when presented with a choice; then, if $\sigma \in \mathfrak{S}_n$ is the unique permutation recording the order~$0,\sigma(1),\sigma(2),\ldots,\sigma(n)$ in which we visited the vertices of $T'$ in the course of this search, relabel~$T'$ by $i \mapsto \sigma^{-1}(i)$ to obtain a rooted plane tree~$T$; lastly, define~$\prec$ by~$i \prec j$ if and only if $\sigma(i) < \sigma(j)$. See Figure~\ref{fig:canonicalpair} for an example of this bijective correspondence. Note crucially that under this correspondence inversions of~$T'$ correspond to inversions of~$T$ with respect to~$\prec$.

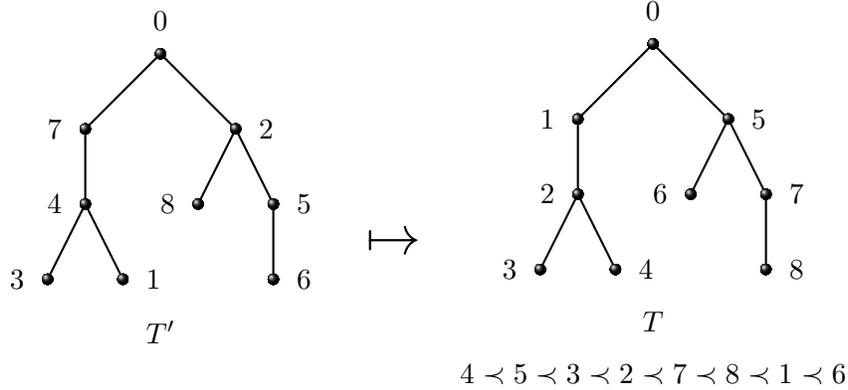
\begin{figure}
\begin{tikzpicture}[scale=1]
	\SetFancyGraph
	\Vertex[LabelOut,Lpos=90, Ldist=.1cm,x=0,y=0]{0}
	\Vertex[LabelOut,Lpos=180, Ldist=.1cm,x=-1,y=-1]{7}
	\Vertex[LabelOut,Lpos=0, Ldist=.1cm,x=1,y=-1]{2}
	\Vertex[LabelOut,Lpos=180, Ldist=.1cm,x=-1,y=-2]{4}
	\Vertex[LabelOut,Lpos=180, Ldist=.1cm,x=0.5,y=-2]{8}
	\Vertex[LabelOut,Lpos=0, Ldist=.1cm,x=1.5,y=-2]{5}
	\Vertex[LabelOut,Lpos=180, Ldist=.1cm,x=-1.5,y=-3]{3}
	\Vertex[LabelOut,Lpos=0, Ldist=.1cm,x=-0.5,y=-3]{1}
	\Vertex[LabelOut,Lpos=0, Ldist=.1cm,x=1.5,y=-3]{6}
	\Edges[style={thick}](0,7)
	\Edges[style={thick}](0,2)
	\Edges[style={thick}](7,4)
	\Edges[style={thick}](2,8)
	\Edges[style={thick}](2,5)
	\Edges[style={thick}](4,3)
	\Edges[style={thick}](4,1)
	\Edges[style={thick}](5,6)
	\node at (0,-3.7) {$T'$};
	\node at (0,-4.4) {};
\end{tikzpicture} \; \; {\parbox{0.3in}{\huge\vspace{-1.5in}$\mapsto$}} \; \begin{tikzpicture}[scale=1]
	\SetFancyGraph
	\Vertex[LabelOut,Lpos=90, Ldist=.1cm,x=0,y=0]{0}
	\Vertex[LabelOut,Lpos=180, Ldist=.1cm,x=-1,y=-1,L={1}]{7}
	\Vertex[LabelOut,Lpos=0, Ldist=.1cm,x=1,y=-1,L={5}]{2}
	\Vertex[LabelOut,Lpos=180, Ldist=.1cm,x=-1,y=-2,L={2}]{4}
	\Vertex[LabelOut,Lpos=180, Ldist=.1cm,x=0.5,y=-2,L={6}]{8}
	\Vertex[LabelOut,Lpos=0, Ldist=.1cm,x=1.5,y=-2,L={7}]{5}
	\Vertex[LabelOut,Lpos=180, Ldist=.1cm,x=-1.5,y=-3,L={3}]{3}
	\Vertex[LabelOut,Lpos=0, Ldist=.1cm,x=-0.5,y=-3,L={4}]{1}
	\Vertex[LabelOut,Lpos=0, Ldist=.1cm,x=1.5,y=-3,L={8}]{6}
	\Edges[style={thick}](0,7)
	\Edges[style={thick}](0,2)
	\Edges[style={thick}](7,4)
	\Edges[style={thick}](2,8)
	\Edges[style={thick}](2,5)
	\Edges[style={thick}](4,3)
	\Edges[style={thick}](4,1)
	\Edges[style={thick}](5,6)
	\node at (0,-3.7) {$T$};
	\node at (0,-4.4) {$4 \prec 5 \prec 3 \prec 2 \prec 7 \prec 8 \prec 1 \prec 6$};
\end{tikzpicture}
\caption{An example of the bijective correspondence between labeled trees and rooted plane trees together with an admissible vertex order. On the left we have an arbitrary labeled tree $T'$ on $9$ vertices. On the right we have the corresponding~$T \in \mathrm{RPT}(9)$ and~$\prec \in \mathrm{AVO}(T)$.} \label{fig:canonicalpair}
\end{figure}

The main result of this paper is a bijective proof of the following theorem.

\begin{thm}\label{thm:main}
For $\mathbf{x} = (x_1,\ldots,x_n) \in \mathbb{N}^n$ we have
\[ \sum_{\alpha \in \mathrm{PF}(\mathbf{x})} q^{\mathrm{rsum}(\alpha)} = \sum_{T \in \mathrm{RPT}(n+1)} \left(\sum_{\prec \in \mathrm{AVO}(T)}q^{\sum_{\substack{i \leq_T j,\\ j \prec i}} \; x_{\mathrm{par}_T(i)+1}}\right) \left(\prod_{i=1}^{n} [x_i]^{\mathrm{outdeg}_T(i-1)}_q\right)\]
where $\mathrm{outdeg}_T(i) := \#\{j \in V(T) \colon i = \mathrm{par}_T(j)\}$ is the number of children of $i$ in $T$.
\end{thm}

Theorem~\ref{thm:main} does relate the reversed sum enumerator for $\mathbf{x}$-parking functions to tree inversions, albeit in a slightly complicated way. (The dependence of the expression on the parent of $i$ for an inversion~$(i,j)$ is a phenomenon going back to Gessel's original definition~\cite{gessel1995enumerative} of his $\kappa$-number.) It is worth comparing Theorems~\ref{thm:kungyan} and~\ref{thm:main}. They both express the reversed sum enumerator for $\mathbf{x}$-parking functions as a sum of Catalan many ``terms.'' The ``main parts'' (products of $q$-numbers) of these terms are the same in the two formulas; however, the ``coefficients'' in front of the main parts are different. Thus Theorems~\ref{thm:kungyan} and~\ref{thm:main} together imply some surprising equalities between expressions involving $q$ and the variables~$x_1,\ldots,x_n$. Indeed, we can see these surprising equalities emerge already in the smallest nontrivial case of~$n:=2$, as we detail in the following example.

\begin{example} \label{ex:n2}
Let $\mathbf{x} := (x_1,x_2) \in \mathbb{N}^2$. First let us compute the reversed sum enumerator for $\mathbf{x}$-parking functions according to the formula in Theorem~\ref{thm:kungyan}. In this case~$\Gamma(2) = \{(2,0),(1,1)\}$. For $\gamma = (2,0)$ we have $\sum_{i=1}^{2}(\gamma_{1} + \cdots+\gamma_{i}- i)x_{i+1} = x_2$ and for $\gamma=(1,1)$ we have $\sum_{i=1}^{2}(\gamma_{1} + \cdots+\gamma_{i}- i)x_{i+1} = 0$. Thus from Theorem~\ref{thm:kungyan} we conclude that
\[ \sum_{\mathcal{\alpha}\in \mathrm{PF}(\mathbf{x})}q^{\mathrm{rsum}(\alpha)} = q^{x_2}[x_1]_q^2 + 2[x_1]_q[x_2]_q  \]
Now let us compute the reversed sum enumerator for $\mathbf{x}$-parking functions according to Theorem~\ref{thm:main}. So we fill out the following table:
\begin{center}
\newcolumntype{C}{ >{\centering\arraybackslash} m{3cm} }
\begin{tabular}{C | C | C | C }
$T \in \mathrm{RPT}(3)$ & $\prec \in \mathrm{AVO}(T)$ & Inversions of $T$ w.r.t.~$\prec$ & $\sum_{\substack{i \leq_T j, \\ j \prec i}} x_{\mathrm{par}_T(i)+1}$ \\
\hline
\begin{tikzpicture}[scale=0.6]
	\SetFancyGraph
	\Vertex[LabelOut,Lpos=90, Ldist=.1cm,x=0,y=1]{0}
	\Vertex[LabelOut,Lpos=270, Ldist=.1cm,x=-1,y=0.25]{1}
	\Vertex[LabelOut,Lpos=270, Ldist=.1cm,x=1,y=0.25]{2}
	\Edges[style={thick}](0,1)
	\Edges[style={thick}](0,2)
\end{tikzpicture} & $2 \prec 1$ & $\varnothing$ & $0$ \\
\begin{tikzpicture}[scale=0.6]
	\SetFancyGraph
	\Vertex[LabelOut,Lpos=0, Ldist=.1cm,x=0,y=1]{0}
	\Vertex[LabelOut,Lpos=0, Ldist=.1cm,x=0,y=0]{1}
	\Vertex[LabelOut,Lpos=0, Ldist=.1cm,x=0,y=-1]{2}
	\Edges[style={thick}](0,1)
	\Edges[style={thick}](1,2)
\end{tikzpicture}  & $1 \prec 2$ & $\varnothing$ & $0$ \\
\begin{tikzpicture}[scale=0.6]
	\SetFancyGraph
	\Vertex[LabelOut,Lpos=0, Ldist=.1cm,x=0,y=1]{0}
	\Vertex[LabelOut,Lpos=0, Ldist=.1cm,x=0,y=0]{1}
	\Vertex[LabelOut,Lpos=0, Ldist=.1cm,x=0,y=-1]{2}
	\Edges[style={thick}](0,1)
	\Edges[style={thick}](1,2)
\end{tikzpicture} & $2 \prec 1$ & $\{(1,2)\}$ & $x_1$ \\
\end{tabular}
\end{center}
Thus from Theorem~\ref{thm:main} we conclude that
\[ \sum_{\mathcal{\alpha}\in \mathrm{PF}(\mathbf{x})}q^{\mathrm{rsum}(\alpha)} = [x_1]_q^2 + (1+q^{x_1})[x_1]_q[x_2]_q.\]
In fact we have $q^{x_2}[x_1]_q^2 + 2[x_1]_q[x_2]_q = [x_1]_q^2 + (1+q^{x_1})[x_1]_q[x_2]_q$ for all $\mathbf{x} = (x_1,x_2) \in \mathbb{N}^2$, so Theorems~\ref{thm:kungyan} and~\ref{thm:main} agree in this case.
\end{example} 

\begin{example} \label{ex:n3}
For $n:=3$, Theorems~\ref{thm:kungyan} and~\ref{thm:main} together imply that
{ \begin{gather*}
q^{2x_2+x_3}[x_1]_q^3 + 3q^{x_2+x_3}[x_1]_q^2[x_2]_q + 3q^{x_2}[x_1]_q^2[x_3]_q + 3q^{x_3}[x_1]_q[x_2]_q^2 + 6[x_1]_q[x_2]_q[x_3]_q = \\
[x_1]_q^3 + (1+2q^{x_1})[x_1]_q^2[x_2]_q + (2+q^{x_1})[x_1]_q^2[x_3]_q + (1+q^{x_1}+q^{2x_1})[x_1]_q[x_2]_q^2 \\
+ (1+q^{x_1}+q^{x_2}+q^{2x_1}+q^{x_1+x_2}+q^{2x_1 + x_2})[x_1]_q[x_2]_q[x_3]_q
\end{gather*}}
for all $\mathbf{x} = (x_1,x_2,x_3) \in \mathbb{N}^3$, an equality which is true but not obvious.
\end{example}

By substituting $q:=1$ in Theorem~\ref{thm:main} we obtain the following corollary.

\begin{cor} \label{cor:main}
For all~$\mathbf{x} = (x_1,\ldots,x_{n}) \in \mathbb{N}^n$, the number of $\mathbf{x}$-parking functions is
\[ n! \sum_{T \in \mathrm{RPT}(n+1)} \left( \prod_{i=1}^{n} \frac{x_{i}^{\mathrm{outdeg}_T(i-1)}}{\mathrm{outdeg}_T(i-1)!} \right).\]
\end{cor}

Corollary~\ref{cor:main} and Theorem~\ref{thm:pitmanstanley} are easily seen to be equivalent, as we explain in Remark~\ref{rem:corequivalent} below. Moreover, the proof of Theorem~\ref{thm:pitmanstanley} appearing in~\cite{stanley2002polytope} is much more straightforward than the DFS-burning algorithm we use to establish Corollary~\ref{cor:main}. We stress that the expression for the reversed sum enumerator in Theorem~\ref{thm:main} is our main result. In particular we believe, as suggested by Examples~\ref{ex:n2} and~\ref{ex:n3}, that Theorems~\ref{thm:kungyan} and~\ref{thm:main} are not equivalent in any obvious way and thus Theorem~\ref{thm:main} is a genuinely new expression for the reversed sum enumerator, one which incorporates tree inversions. A feature of Theorem~\ref{thm:main} which Theorem~\ref{thm:kungyan} lacks is that, thanks to the correspondence detailed above between labeled trees $T'$ and pairs $(T,\prec)$ of rooted plane trees with admissible vertex orders, Theorem~\ref{thm:main} recovers the foundational result~(\ref{eqn:classical}) of Kreweras when $\mathbf{x} := (1,1,\ldots,1)$.

The moral of this paper is that while $G$-parking functions and $\mathbf{x}$-parking functions are two largely orthogonal generalizations of classical parking functions, techniques from the former can fruitfully be applied to the latter, especially when it comes to the relationship between trees and parking functions.

\medskip

\noindent {\bf Acknowledgements}: This research was carried out at MIT as part of the RSI summer mathematics research program for high-school students. The second author was the mentor of the first author. We thank David Perkinson, Alex Postnikov, and Richard Stanley for some helpful comments. We also thank an anonymous referee for very many helpful comments, especially with regards to the exposition of the paper and with pointers to the (vast) parking function literature. The second author was partially supported by NSF grant~\#1122374.

\section{The overlap between graphical and vector parking functions}

In this section we review graphical parking functions with the aim of understanding their overlap with vector parking functions. Beyond Definition~\ref{def:gpf}, none of this section is necessary for the proof of our main result (Theorem~\ref{thm:main}). But Theorem~\ref{thm:gpfandxpf} below does put our work in context because it shows that graphical and vector parking functions are genuinely different objects.

Throughout this paper we will work with (finite) \emph{multigraphs}, i.e.,~finite, undirected graphs for which multiple edges, but not loops, are allowed. Formally, we view a multigraph $G$ as consisting of its vertex set~$V(G)$ together with its \emph{edge-weight function}~$\omega_{G}\colon \binom{V(G)}{2} \to \mathbb{N}$. However, we will also often think about~$G$ as consisting of its vertex set $V(G)$ together with its multiset~$E(G)$ of edges: so $E(G)$ contains~$\omega_G(e)$ many copies of $e$ for each~$e \in \binom{V(G)}{2}$.  By~$\#E(G) := \sum_{e \in \binom{V(G)}{2}} \omega_G(e)$ we will mean the number of edges of $G$ counted with multiplicity. Any multigraph $G$ under consideration will come with a distinguished \emph{root} vertex. All constructions that follow will implicitly depend on this choice of root but we will suppress the dependence in our notation. We use~$\widetilde{V}(G)$ to denote the set of non-root vertices of $G$. For the rest of this section, $G$ is a fixed multigraph with~$V(G) := \{0,1,\ldots,n\}$ (where $n\geq 1$) and root $0$.

\begin{definition} \label{def:gpf}
A \emph{$G$-parking function} is a formal sum $\alpha = \sum_{i \in \widetilde{V}(G)} \alpha_i [i] \in \mathbb{N}\widetilde{V}(G)$ of non-root vertices such that for any $\varnothing \neq U \subseteq \widetilde{V}(G)$ there is~$i \in U$ with~\mbox{$\alpha_i \leq \mathrm{deg}^G_U(i) - 1$} where~$\mathrm{deg}^G_U(i) := \sum_{j \in V(G) - U} \omega_G(\{i,j\})$. We identify the formal sum $\alpha = \sum_{i \in \widetilde{V}(G)} \alpha_i [i]$ with the sequence $\alpha = (\alpha_1,\ldots,\alpha_n) \in \mathbb{N}^n$ when $\widetilde{V}(G) = \{1,\ldots,n\}$. We denote the set of $G$-parking functions by~$\mathrm{PF}(G)$. The \emph{sum} of~$\alpha \in \mathrm{PF}(G)$ is~$\mathrm{sum}(\alpha) := \sum_{i=1}^{n} \alpha_i$ and the \emph{reversed sum} is~$\mathrm{rsum}(\alpha) := \#E(G) - n - \mathrm{sum}(\alpha)$. 
\end{definition}

We say that $G$ is \emph{connected} if for each~$\varnothing \neq U \subseteq \widetilde{V}(G)$ there is $i \in U$ with $\mathrm{deg}^G_U(i) \geq 1$. Note that $\mathrm{PF}(G) \neq \varnothing$ if and only if~$G$ is connected. Also note that $\mathrm{PF}(G) = \mathrm{PF}(n)$ when~$G = K_{n+1}$ is the complete graph, i.e.,~when $\omega_G(e) = 1$ for all $e \in\binom{V(G)}{2}$. So indeed graphical parking functions are a generalization of classical parking functions. Let us now recall some other basic facts about graphical parking functions. The first fact, which we also mentioned in Section~\ref{sec:intro}, says that the number of~$G$-parking functions is the number of spanning trees of $G$ (weighted by edge multiplicities). A \emph{tree}~$T$ is a connected multigraph with~$\#E(T) = \#V(T)-1$. We will always consider a tree to be rooted at~$0 \in V(T)$. Therefore every tree~$T$ comes with a partial order~$\leq_T$ on~$V(T)$ as described in Section~\ref{sec:intro}. A \emph{spanning subgraph~$H$ of $G$} is a multigraph $H$ with~$V(H) = V(G)$ and~$E(H) \subseteq E(G)$. A \emph{spanning tree $T$ of~$G$} is a spanning subgraph of~$G$ which is a tree. We use~$\mathrm{SPT}(G)$ to denote the set of spanning trees of $G$.

\begin{fact} \label{fact:matrixtree}
We have
\[ \#\mathrm{PF}(G) = \sum_{T \in \mathrm{SPT}(G)} \; \prod_{e \in E(T)} \omega_G(e). \]
Thus via the Matrix-Tree Theorem (see~\cite[\S5]{moon1969counting}) we conclude that~$\#\mathrm{PF}(G)$ is also equal to the determinant of the reduced Laplacian~$\widetilde{\Delta}(G)$ of~$G$.
\end{fact}
\begin{proof}
This is very well-known; see for instance~\cite[Theorem 2.1]{postnikov2004trees}. At any rate, it will follow from Theorem~\ref{thm:gpfrsum} below by substituting~$q:=1$.
\end{proof}

The next fact concerns maximal $G$-parking functions. The monoid $\mathbb{N}^n$ has a natural partial order whereby for $\alpha = (\alpha_1,\ldots,\alpha_n), \alpha' = (\alpha'_1,\ldots,\alpha'_n) \in \mathbb{N}^n$ we write $\alpha \leq \alpha'$ if and only if~$\alpha_i \leq \alpha'_i$ for all $1 \leq i \leq n$. A \emph{maximal} $G$-parking function is a $G$-parking function that is maximal among $G$-parking functions with respect to this partial order. We use~$\mathrm{MPF}(G)$ for the set of maximal $G$-parking functions.

\begin{fact} \label{fact:max}
For $\alpha \in \mathbb{N}^n$, we have $\alpha \in \mathrm{PF}(G)$ if and only if $\alpha \leq \alpha'$ for some $\alpha' \in \mathrm{MPF}(G)$.
\end{fact}
\begin{proof}
This is an easy consequence of Definition~\ref{def:gpf}.
\end{proof}

The final fact concerns another, extremely useful, combinatorial model for $\mathrm{MPF}(G)$ in terms of orientations. An \emph{orientation} of $G$ is a subset $\mathcal{O} \subseteq V(G)^2$ such that
\begin{itemize}
\item $(i,i) \notin \mathcal{O}$ for any $i$;
\item $(i,j) \in \mathcal{O}$ implies $\omega_G(\{i,j\}) \geq 1$ for all $\{i,j\} \in \binom{V(G)}{2}$;
\item for each $\{i,j\} \in \binom{V(G)}{2}$ with $\omega_G(\{i,j\}) \geq 1$, $\# \{(i,j),(j,i)\} \cap \mathcal{O} = 1$.
\end{itemize}
Given an orientation $\mathcal{O}$ of $G$ and $i \in V(G)$, the \emph{indegree of $i$ with respect to $\mathcal{O}$} is~$\mathrm{indeg}_{\mathcal{O}}(i) := \sum_{(j,i) \in \mathcal{O}} \omega_G(\{i,j\})$. We say $i \in V(G)$ is a \emph{source of $\mathcal{O}$} if~$\mathrm{indeg}_{\mathcal{O}}(i) = 0$ and we say it is a \emph{sink of $\mathcal{O}$} if~$\mathrm{indeg}_{\mathcal{O}}(i) = \mathrm{deg}_G(i)$. Here the \emph{degree of $i$ in $G$} is~$\mathrm{deg}_G(i) := \sum_{i \neq j \in V(G)} \omega_G(\{i,j\}) = \mathrm{deg}^{G}_{\{i\}}(i)$. Finally, we say $\mathcal{O}$ is \emph{acyclic} if there does not exist a sequence $(i_1,i_2),(i_2,i_3),\ldots,(i_{k},i_{k+1})$ with $(i_{j},i_{j+1}) \in \mathcal{O}$ for all~$1 \leq j \leq k$ and~$i_1 = i_{k+1}$. Let $\mathcal{A}(G)$ be the set of acyclic orientations of $G$ with unique source~$0$.

\begin{fact} \label{fact:orientations}
The map $\mathcal{O} \mapsto  (\mathrm{indeg}_\mathcal{O}(1)-1,\mathrm{indeg}_\mathcal{O}(2)-1,\ldots,\mathrm{indeg}_\mathcal{O}(n)-1)$ is a bijection between $\mathcal{A}(G)$ and~$\mathrm{MPF}(G)$.
\end{fact}
\begin{proof}
This is also well-known among those who study graphical parking functions; see~\cite[Theorem 3.1]{benson2010gparking}. The proof uses Dhar's burning algorithm.
\end{proof}

The symmetric group $\mathfrak{S}_n$ acts on $\mathbb{N}^n$ by~$\sigma(\alpha) := (\alpha_{\sigma^{-1}(1)},\ldots,\alpha_{\sigma^{-1}(n)})$ for~ $\sigma \in \mathfrak{S}_n$ and~$\alpha = (\alpha_1,\ldots,\alpha_n) \in \mathbb{N}^n$. We will now classify the multigraphs $G$ for which $\mathrm{PF}(G)$ is invariant under this action of~$\mathfrak{S}_n$; an immediate consequence will be a classification of those multigraphs whose parking functions are also the parking functions of some vector. To that end let us define some families of multigraphs. A \emph{cycle} $C$ is a connected multigraph with~$\mathrm{deg}_C(i) = 2$ for all $i \in V(C)$. For $a \geq 1$ let us say $G$ is an \emph{$a$-cycle} if there is some cycle $C$ with $V(C) = V(G)$ and $\omega_G(e) = a\cdot \omega_C(e)$ for all $e \in \binom{V(G)}{2}$. Similarly, for $a \geq 1$ let us say $G$ is an \emph{$a$-tree} if there is some tree~$T$ with $V(T) = V(G)$ and~$\omega_G(e) = a\cdot \omega_T(e)$ for all $e \in \binom{V(G)}{2}$. Finally, for~$a,b \geq 1$ let~$K_{n+1}^{a,b}$ be the multigraph with vertex set~$V(K_{n+1}^{a,b}) := \{0,1,\ldots,n\}$ and edge-weight function 
\[ \omega_{K_{n+1}^{a,b}}(\{i,j\}) := \begin{cases} a &\textrm{if $i=0$ or $j=0$} \\ b &\textrm{otherwise}.\end{cases} \]

\begin{thm} \label{thm:gpfandxpf}
If $\mathrm{PF}(G)$ is invariant under the action of $\mathfrak{S}_n$ on $\mathbb{N}^n$ then one of the following holds:
\begin{itemize}
\item $G$ is an $a$-tree and $\mathrm{PF}(G) = \mathrm{PF}( (a,\overbrace{0,0,\ldots,0}^{n-1}) )$ for some $a\geq 1$;
\item $G$ is an $a$-cycle and $\mathrm{PF}(G) = \mathrm{PF}( (a,\overbrace{0,0,\ldots,0}^{n-2},a) )$ for some $a\geq 1$;
\item $G = K_{n+1}^{a,b}$ and $\mathrm{PF}(G) = \mathrm{PF}( (a,\overbrace{b,b,\ldots,b}^{n-1}) )$ for some $a,b\geq 1$.
\end{itemize}
On the other hand, if $\mathrm{PF}(G)$ is not invariant under the action of the symmetric group then~$\mathrm{PF}(G) \neq \mathrm{PF}(\mathbf{x})$ for any $\mathbf{x} \in \mathbb{N}^{n}$.
\end{thm}

\begin{proof}
From Definition~\ref{def:xpf} it immediately follows that  for all $\mathbf{x} \in \mathbb{N}^n$,  $\mathrm{PF}(\mathbf{x})$ is invariant under the action of the symmetric group so the last sentence is certainly true. 

So let us consider when $\mathrm{PF}(G)$ can be invariant under~$\mathfrak{S}_n$. We claim that $\mathrm{PF}(G)$ is invariant under the action of the symmetric group if and only if $\mathrm{MPF}(G)$ is invariant under the action of the symmetric group. Fact~\ref{fact:orientations} implies that all maximal~$G$-parking functions have the same sum, which means that if $\alpha \in \mathrm{MPF}(G)$ and~$\sigma(\alpha) \in \mathrm{PF}(G)$ then~$\sigma(\alpha) \in \mathrm{MPF}(G)$. The other direction follows from Fact~\ref{fact:max}.

From now on we focus on maximal $G$-parking functions. We want to show that if~$\mathrm{MPF}(G)$ is invariant under the action of the symmetric group then $G$ is an $a$-tree or $G$ is an $a$-cycle or~$G = K_{n+1}^{a,b}$. We prove this by induction on $n := \#V(G) - 1$. The base case of $n=1$ is trivial. So suppose that~$n \geq 2$ and $\mathrm{MPF}(G)$ is invariant under~$\mathfrak{S}_n$. For~$i \in V(G)$ let $G - i$ denote the multigraph obtained by \emph{deleting}~$i$: this is the multigraph on vertex set~$V(G) - \{i\}$ and with edge-weight function equal to~$\omega_G$ restricted to~$\binom{V(G) - \{i\}}{2}$. A \emph{cut vertex} of~$G$ is a vertex for which $G-i$ is not connected. Define~$\mathcal{I}(G) \subseteq V(G)$ to be the set of vertices of $G$ that are not cut vertices and additionally are not equal to~$0$. It is a simple fact that every connected multigraph on at least two vertices has at least two vertices that are not cut vertices. Thus $\mathcal{I}(G)$ is in fact nonempty. The inductive step in the argument will be to consider $G-i$ and its set of parking functions for some~$i \in \mathcal{I}(G)$.

For any~$1 \leq i \leq n$ let~$\mathrm{MPF}(G;i)$ be the subset of~$\alpha = (\alpha_1,\ldots,\alpha_n) \in \mathrm{MPF}(G)$ for which $\alpha_i$ is maximized. We claim that for~$i \in \mathcal{I}(G)$ the map $\beta \mapsto \beta + (\mathrm{deg}_G(i) - 1)[i]$ is a bijection between $\mathrm{MPF}(G - i)$ and $\mathrm{MPF}(G;i)$ (when we continue to view $0$ as the root of $G-i$). Why is this? First, note that an equivalent description of $\mathcal{A}(G)$ is as the set of acyclic orientations $\mathcal{O}$ of $G$ for which there is a path from~$0$ to~$i$ for every~$i \in \widetilde{V}(G)$. Here a \emph{path in $\mathcal{O}$ from $s$ to $t$} is a sequence~$(i_1,i_2),(i_2,i_3),\ldots,(i_k,i_{k+1})$ with $i_1=s$, $i_{k+1} = t$, and $(i_j,i_{j+1}) \in \mathcal{O}$ for all~$1\leq j \leq k$. For $i \in V(G)$ let $\mathcal{A}(G;i)$ be the set of those orientations~$\mathcal{O} \in \mathcal{A}(G)$ for which $i$ is a sink. The alternate description of~$\mathcal{A}(G)$ implies that for any~$1\leq i \leq n$ there is a bijection between~$\mathcal{A}(G;i)$ and~$\mathcal{A}(G - i)$ given by $\mathcal{O} \mapsto \mathcal{O} \cap V(G-i)^2$. Next, note that~$\mathcal{A}(G - i) = \varnothing$ if and only if $i$ is a cut vertex (one way to see this is via Fact~\ref{fact:orientations}: we have~$\mathrm{MPF}(G-i) \neq \varnothing$ if and only if $\mathrm{PF}(G-i) \neq \varnothing$ and, as mentioned earlier, $\mathrm{PF}(G-i) \neq \varnothing$ if and only if~$G-i$ is connected). Thus~$\mathcal{A}(G;i)$ is nonempty for all $i\in \mathcal{I}(G)$. But as long as $\mathcal{A}(G;i)$ is nonempty, Fact~\ref{fact:orientations} says that the maximal value of $\alpha_i$ among~$(\alpha_1,\ldots,\alpha_n) \in \mathrm{MPF}(G)$ must be~$\mathrm{deg}_G(i)-1$. Moreover, in this case the elements of~$\mathrm{MPF}(G)$ that achieve this maximal~$\alpha_i$ are exactly the ones whose corresponding orientations (under the bijection of Fact~\ref{fact:orientations}) belong to~$\mathcal{A}(G;i)$. That the map~$\beta \mapsto \beta + (\mathrm{deg}_G(i) - 1)[i]$ is a bijection between $\mathrm{MPF}(G - i)$ and $\mathrm{MPF}(G;i)$ for~$i \in \mathcal{I}(G)$ then follows from the fact that $\mathcal{O} \mapsto \mathcal{O} \cap V(G-i)^2$ is a bijection between~$\mathcal{A}(G;i)$ and~$\mathcal{A}(G-i)$  (together with Fact~\ref{fact:orientations}).

There are two important consequences of the construction in the previous paragraph. The first consequence is that
\begin{equation*} \label{eqn:degrees}
\mathrm{deg}_G(i) = \mathrm{deg}_G(j) \textrm{ for all $i,j \in \mathcal{I}(G)$}. \tag{$\star$}
\end{equation*}
Why does~(\ref{eqn:degrees}) hold? Because, as mentioned above, if $i \in \mathcal{I}(G)$ then the maximal value of $\alpha_i$ among~$(\alpha_1,\ldots,\alpha_n) \in \mathrm{MPF}(G)$ is $\mathrm{deg}_G(i)-1$; but since $\mathrm{MPF}(G)$ is invariant under~$\mathfrak{S}_n$, this maximal value must be the same for all $1 \leq i \leq n$. The second consequence is that
\begin{equation*} \label{eqn:deletion}
\textrm{$G-i$ is an $a$-cycle or $G$ is an $a$-tree or $G-i = K_n^{a,b}$ for all $i \in \mathcal{I}(G)$}. \tag{$\star\star$}
\end{equation*}
Why does~(\ref{eqn:deletion}) hold? The invariance of $\mathrm{MPF}(G)$ under~$\mathfrak{S}_n$ implies that~$\mathrm{MPF}(G;i)$ is invariant under the subgroup of~$\mathfrak{S}_n$ consisting of permutations that fix~$i$ (this subgroup is of course isomorphic to~$\mathfrak{S}_{n-1}$.) But then the bijection between~$\mathrm{MPF}(G;i)$ and~$\mathrm{MPF}(G - i)$ from the last paragraph implies that $\mathrm{MPF}(G - i)$ is invariant under~$\mathfrak{S}_{n-1}$. By our inductive hypothesis we conclude~(\ref{eqn:deletion}).

As we pointed out earlier, $\mathcal{I}(G)$ is nonempty. Without loss of generality let us assume that~$n \in \mathcal{I}(G)$. By~(\ref{eqn:deletion}) we know that~$G - n$ is an $a$-tree or $G-n$ is an $a$-cycle or~$G - n = K_{n}^{a,b}$. We will consider each of these cases (although not in that order). First we define a bit of terminology that will aid in our analysis of these cases. For a multigraph~$H$ and~$i,j \in V(H)$, we say that $i$ is \emph{adjacent} to $j$ in $H$ if~$\omega_H(\{i,j\}) \neq 0$. We say $H$ contains a \emph{cycle of adjacent vertices} if there is a sequence~$i_1,i_2,\ldots,i_k,i_{k+1}$ of vertices for some $k \geq 3$ such that $i_1 = i_{k+1}$, $i_j$ and $i_{j+1}$ are adjacent for all $1 \leq j \leq k$, and $i_{j} \neq i_{j'}$ for any $1 \leq j,j' \leq k$. Note that if $H$ is an $a$-tree then~$H$ does not contain a cycle of adjacent vertices. Also note that if $H$ is an $a$-cycle and $\#V(H) \geq 3$ then every vertex of~$H$ is adjacent to exactly two vertices. And if $H = K^{a,b}_{\#V(H)+1}$ then every vertex of~$H$ is adjacent to every other vertex.

\medskip

\noindent \fbox{{\bf Case I}: $G - n$ is an $a$-cycle with $a \geq 1$ and $n \geq 4$.}

First suppose that $n$ is adjacent to only one $i \in V(G)$. We must have $\omega_G(\{i,n\}) = 2a$ because of~(\ref{eqn:degrees}). Take any~$1 \leq j < n$ with $i \neq j$ and observe that $G - j$ is not an $a'$-tree (because it has an edge of weight $\omega_{G}(\{i,n\}) = 2a > a$ as well as an edge of weight $a$) or an $a'$-cycle or equal to $K^{a',b}_{n}$ (because it has a vertex adjacent to only one other vertex), contradicting~(\ref{eqn:deletion}). Thus $n$ must be adjacent to more than one vertex. Now suppose that $n$ is adjacent to at least two other vertices. Then note that~$\mathcal{I}(G) = \widetilde{V}(G)$ and so because of~(\ref{eqn:degrees}) we must have $\omega_G(\{i,n\}) = \omega_G(\{j,n\})$ for all $1 \leq i,j < n$. Let~$1 \leq i < n$ be such that~$i$ is adjacent to $0$. We can see that $G - i$ is not an $a'$-tree (because it contains a cycle of adjacent vertices) or an $a'$-cycle (because it has a vertex adjacent to at least three other vertices) or $K^{a',b}_{n}$ (because it has at least two vertices that are not adjacent), contradicting~(\ref{eqn:deletion}). Therefore this case is actually impossible.

\medskip

\noindent \fbox{{\bf Case II}: $G - n = K^{a,b}_{n}$ with $a,b \geq 1$ and $n \geq 3$. }

This case includes when $G-n$ is an $a$-cycle and $n=3$. Suppose that~$n$ is adjacent to only one $i \in V(G)$. Then $\omega(\{i,n\}) = a(n-1) + b$ because of~(\ref{eqn:degrees}). Take any~$1 \leq j < n$ with~$i \neq j$ and observe that $G - j$ is not an $a'$-tree (because it contains an edge of weight $\omega(\{i,n\}) = a(n-1) + b > b$ as well as an edge of weight~$b$) or an~$a'$-cycle or~$K^{a',b'}_{n}$ (because it has a vertex adjacent to only one other vertex), contradicting~(\ref{eqn:deletion}). Thus~$n$ must be adjacent to more than one vertex. Now suppose that $n$ is adjacent to at least two other vertices. Then note that~$\mathcal{I}(G) = \widetilde{V}(G)$ and so because of~(\ref{eqn:degrees}) we must have $\omega_G(\{i,n\}) = \omega_G(\{j,n\})$ for all $1 \leq i,j < n$. Let $1 \leq i < n$. We see that~$G - i$ cannot be an $a'$-tree (because it contains a cycle of adjacent vertices) and~$G-i$ cannot be an $a'$-cycle with $n \geq 4$ (because any two of its vertices are adjacent). So from~(\ref{eqn:deletion}) we conclude $G-i = K^{a',b'}_{n}$. Moreover, by considering edges that belong to both~$G-i$ and~$G-n$ we must have~$a=a'$ and $b=b'$. This implies in particular that~$\omega_G(\{0,n\}) = a$ and~$\omega_G(\{j,n\}) = b$ for all~$1 \leq j < n$. Thus~$G = K^{a,b}_{n+1}$.

\medskip

\noindent \fbox{{\bf Case III}: $G - n$ is an $a$-tree with $a \geq 1$. }

This case includes when $G-n$ is an $a$-cycle and $n=2$, as well as~$G-n= K_2^{a,b}$. First suppose $n$ is adjacent to only one vertex. Then~$\mathcal{A}(G)$ has only one element and Fact~\ref{fact:orientations} implies that the only way that $\mathrm{MPF}(G)$ is invariant under~$\mathfrak{S}_n$ is if $G$ is an $a$-tree. Thus, suppose from now on that $n$ is adjacent to at least two vertices. If $n=2$ then we can see from~(\ref{eqn:degrees}) that $G = K^{a,b}_{3}$ since~$\mathcal{I}(G) = \{1,2\}$. So now suppose $n=3$. First suppose that $G - 3$ is
\begin{equation*} \label{eqn:g3case1} 
\begin{multlined}
\begin{tikzpicture}[scale=0.9,>=latex,auto,rotate=90]
	\SetFancyGraph
	\Vertex[LabelOut,Lpos=270, Ldist=0cm,x=0,y=1]{1}
	\Vertex[LabelOut,Lpos=270, Ldist=0cm,x=0,y=0]{0}%
	\Vertex[LabelOut,Lpos=270, Ldist=0cm,x=0,y=-1]{2}
	\Edges[style={thick},label=$a$](1,0)
	\Edges[style={thick},label=$a$](0,2)
\end{tikzpicture}
\end{multlined} \tag{\dag}
\end{equation*}
(Here we label an edge $e \in E(G)$ by $\omega_G(e)$, with an edge absent if it has weight zero.) With $G-3$ as in~(\ref{eqn:g3case1}) we have~$\mathcal{I}(G) = \{1,2,3\}$. Because of~(\ref{eqn:degrees}) it is impossible that~$3$ is adjacent to only~$1$ and~$0$, or to only~$2$ and~$0$. And because of~(\ref{eqn:degrees}) if~$3$ is adjacent to only~$1$ and~$2$ then~$G = C_4^a$. Lastly, let us rule out the possibility that $G$ looks like
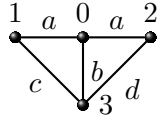
\begin{center}
\begin{tikzpicture}[scale=0.9,>=latex,auto,rotate=90]
	\SetFancyGraph
	\Vertex[LabelOut,Lpos=90, Ldist=0cm,x=0,y=1]{1}
	\Vertex[LabelOut,Lpos=90, Ldist=0cm,x=0,y=0]{0}%
	\Vertex[LabelOut,Lpos=90, Ldist=0cm,x=0,y=-1]{2}
	\Vertex[LabelOut,Lpos=0, Ldist=0cm,x=-1,y=0]{3}
	\Edges[style={thick},label=$a$](1,0)
	\Edges[style={thick},label=$a$](0,2)
	\Edges[style={thick},label=$b$](0,3)
	\Edges[style={thick},label=$c$](3,1)
	\Edges[style={thick},label=$d$](2,3)
\end{tikzpicture}
\end{center}
for some $b,c,d \geq 1$. By applying~(\ref{eqn:deletion}) to~$1$, we see that we must have $b=a$. But then~$G$ cannot satisfy~(\ref{eqn:degrees}) because $\mathrm{deg}_G(3) = a+c+d>a+c=\mathrm{deg}_G(1)$. So indeed it is not possible for~$3$ to be adjacent to $0$, $1$, and $2$ when $G-3$ is as in~(\ref{eqn:g3case1}). We have considered all possibilities for $G-3$ as in~(\ref{eqn:g3case1}). Next suppose that $G - 3$ is
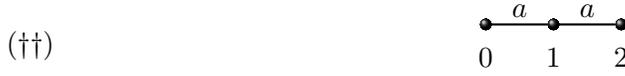
\begin{equation*} \label{eqn:g3case2}
\begin{multlined}
\begin{tikzpicture}[scale=0.9,>=latex,auto,rotate=90]
	\SetFancyGraph
	\Vertex[LabelOut,Lpos=270, Ldist=.1cm,x=0,y=1]{0}
	\Vertex[LabelOut,Lpos=270, Ldist=.1cm,x=0,y=0]{1}%
	\Vertex[LabelOut,Lpos=270, Ldist=.1cm,x=0,y=-1]{2}
	\Edges[style={thick},label=$a$](0,1)
	\Edges[style={thick},label=$a$](1,2)
\end{tikzpicture}
\end{multlined}
\tag{\dag\dag}
\end{equation*}
Because of~(\ref{eqn:degrees}) it is impossible  that $3$ is adjacent only to $1$ and~$0$. And because of~(\ref{eqn:degrees}) if~$3$ is adjacent to only $0$ and $2$ then $G = C_4^a$. Let us rule out the possibility that $G$ looks like
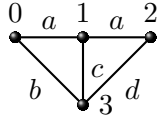
\begin{center}
\begin{tikzpicture}[scale=0.9,>=latex,auto,rotate=90]
	\SetFancyGraph
	\Vertex[LabelOut,Lpos=90, Ldist=0cm,x=0,y=1]{0}
	\Vertex[LabelOut,Lpos=90, Ldist=0cm,x=0,y=0]{1}%
	\Vertex[LabelOut,Lpos=90, Ldist=0cm,x=0,y=-1]{2}
	\Vertex[LabelOut,Lpos=0, Ldist=0cm,x=-1,y=0]{3}
	\Edges[style={thick},label=$a$](0,1)
	\Edges[style={thick},label=$a$](1,2)
	\Edges[style={thick},label=$c$](1,3)
	\Edges[style={thick},label=$b$](3,0)
	\Edges[style={thick},label=$d$](2,3)
\end{tikzpicture}
\end{center}
for some $b,c,d \geq 1$. By applying~(\ref{eqn:deletion}) to~$2$, we see that we must have $b=a$. But then~$G$ cannot satisfy~(\ref{eqn:degrees}) because $\mathrm{deg}_G(3) = a+c+d>a+d=\mathrm{deg}_G(2)$. So indeed it is not possible for~$3$ to be adjacent to $0$, $1$, and $2$ when as in~(\ref{eqn:g3case2}). Lastly, let us rule out the possibility that $G$ looks like
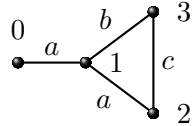
\begin{center}
\begin{tikzpicture}[scale=0.9,>=latex,auto,rotate=90]
	\SetFancyGraph
	\Vertex[LabelOut,Lpos=90, Ldist=.1cm,x=0,y=1]{0}
	\Vertex[LabelOut,Lpos=0, Ldist=.1cm,x=0,y=0]{1}
	\Vertex[LabelOut,Lpos=0, Ldist=.1cm,x=-0.75,y=-1]{2}
	\Vertex[LabelOut,Lpos=0, Ldist=.1cm,x=0.75,y=-1]{3}
	\Edges[style={thick},label=$a$](0,1)
	\Edges[style={thick},label=$a$](2,1)
	\Edges[style={thick},label=$c$](3,2)
	\Edges[style={thick},label=$b$](1,3)
\end{tikzpicture}
\end{center}
for some~$b,c \geq 1$: if it did, then for all~$\mathcal{O} \in \mathcal{A}(G)$ we have $\mathrm{indeg}_{\mathcal{O}}(1) = a < a + c$ but there is some~$\mathcal{O}' \in \mathcal{A}$ with~$\mathrm{indeg}_{\mathcal{O}'}(2) = a+c$ which by Fact~\ref{fact:orientations} directly implies that~$\mathrm{MPF}(G)$ is not invariant under $\mathfrak{S}_n$. We have considered all possibilities for $G-3$ as in~(\ref{eqn:g3case2}). Up to isomorphism, $G- n$ can only look like~(\ref{eqn:g3case1}) or~(\ref{eqn:g3case2}) when $n = 3$.

So finally let us suppose that~$n \geq 4$. Let $T$ be the tree such that $\omega_{G - n} = a \cdot \omega_T$. Observe that $\mathcal{I}(G- n)$ is the subset of elements in~$\{1,\ldots,n-1\}$ that are maximal with respect to~$\leq_T$, so in particular $\mathcal{I}(G- n)$ is nonempty. Note also that~$\mathcal{I}(G- n) \subseteq \mathcal{I}(G)$ because $n$ is adjacent to at least two vertices. First suppose that~$n$ is adjacent to three or more vertices. Then for any~$i \in \mathcal{I}(G- n)$ we will get that~$G - i$ is not~$K_n^{a',b'}$ (because it contains two vertices that are not adjacent) and is not an~$a'$-tree (because it contains a cycle of adjacent vertices). Moreover, by our analysis of~Case~I we also know that~$G - i$ cannot be an $a'$-cycle for any $i \in \mathcal{I}(G)$. This contradicts~(\ref{eqn:deletion}). So it is not possible that $n$ is adjacent to three or more vertices. Next suppose that there is~$i \in \mathcal{I}(G- n)$ which is not adjacent to~$n$. Then we will again have that~$G - i$ is not~$K_n^{a',b'}$ (because it contains two vertices that are not adjacent) and is not an~$a'$-tree (because it contains a cycle of adjacent vertices) and is not an~$a'$-cycle (by appealing to our analysis of~Case~I), contradicting~(\ref{eqn:deletion}). We conclude that $n$ is adjacent to exactly two vertices; more specifically, $n$ is adjacent to all of the vertices in~$\mathcal{I}(G- n)$ (which contains either one or two vertices) and at most one other vertex (in the case where~$\#\mathcal{I}(G-n) = 1$). Therefore, because $\mathcal{I}(G- n)$ is the set of maximal elements of~$T$ with respect to $\leq_T$, and also by applying~(\ref{eqn:deletion}) to each~$i \in \mathcal{I}(G)$ to check that $G-i$ is an $a$-tree, we see that~$G$ looks like
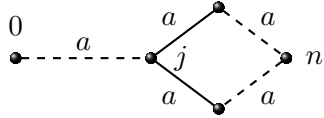
\begin{center}
\begin{tikzpicture}[scale=0.9,>=latex,auto,rotate=90]
	\SetFancyGraph
	\Vertex[LabelOut,Lpos=90, Ldist=.1cm,x=0,y=2]{0}
	\Vertex[LabelOut,Lpos=0, Ldist=.1cm,x=0,y=0]{j}
	\Vertex[NoLabel,x=-0.75,y=-1]{b}
	\Vertex[NoLabel,x=0.75,y=-1]{c}
	\Vertex[LabelOut,Lpos=0, Ldist=.05cm,,x=0,y=-2]{n}
	\Edges[style={thick,dashed},label=$a$](0,j)
	\Edges[style={thick},label=$a$](b,j)
	\Edges[style={thick,dashed},label=$a$](n,b)
	\Edges[style={thick,dashed},label=$a$](c,n)
	\Edges[style={thick},label=$a$](j,c)
\end{tikzpicture}
\end{center}
(The dotted lines indicate paths of edges of weight~$a$.) Suppose vertex $j$ above is not $0$: then for all $\mathcal{O} \in \mathcal{A}(G)$ we have $\mathrm{indeg}_{\mathcal{O}}(j) = a < 2a$ but there is~$\mathcal{O}' \in \mathcal{A}$ with $\mathrm{indeg}_{\mathcal{O}'}(n) = 2a$ which by Fact~\ref{fact:orientations} directly implies that $\mathrm{MPF}(G)$ is not invariant under~$\mathfrak{S}_n$. Thus we must have that~$j = 0$ and that $G$ is an $a$-cycle.

\medskip

To finish the proof, it is straightforward to verify that in all these cases where~$\mathrm{PF}(G)$ is invariant under~$\mathfrak{S}_n$ it is equal to $\mathrm{PF}(\mathbf{x})$ for the claimed~$\mathbf{x} = (x_1,\ldots,x_n) \in \mathbb{N}^n$. For this purpose it is useful to observe that $\alpha \in \mathbb{N}^n$ is an $\mathbf{x}$-parking function if and only if~$\alpha \leq \alpha'$ for some maximal $\mathbf{x}$-parking function $\alpha'$. Thus, thanks to Fact~\ref{fact:max}, we need only check that the maximal $\mathbf{x}$-parking functions are the same as the maximal $G$-parking functions for the appropriate $G$ and $\mathbf{x}$. The maximal $G$-parking are easy to understand because of Fact~\ref{fact:orientations}, while the maximal $\mathbf{x}$-parking functions are merely the permutations of~$(x_1-1,x_1+x_2-1,\ldots,x_1+x_2+\cdots+x_n-1)$.
\end{proof}

\begin{remark}
There are simple product formulas for $\#\mathrm{PF}(\mathbf{x})$ for all of the~$\mathbf{x}$ appearing in~Theorem~\ref{thm:gpfandxpf}. Namely,
\begin{itemize}
\item $\#\mathrm{PF}( (a,\overbrace{0,0,\ldots,0}^{n-1}) ) = a^n$;
\item $\#\mathrm{PF}( (a,\overbrace{0,0,\ldots,0}^{n-2},a) ) = (n+1)a^n$;
\item $\#\mathrm{PF}( (a,\overbrace{b,b,\ldots,b}^{n-1}) = a(a+nb)^{n-1}$.
\end{itemize}
The first two formulas are trivial; for the third see~\cite[Equation~(7)]{stanley2002polytope} or~\cite[Theorem~1]{yan2000enumeration}. 
\end{remark}

\section{The multigraph DFS-burning algorithm}

In this section we extend the work of~\cite{perkinson2013gparking} to multigraphs. Thus for a multigraph~$G$ we relate the reversed sum enumerator of $G$-parking functions to the $\kappa$-number (generalized inversion number) enumerator of spanning trees of $G$. We need to define these $\kappa$-numbers. Let~$G$ be a fixed connected multigraph with $V(G) := \{0,1,\ldots,n\}$ and root~$0$. Following Gessel~\cite{gessel1995enumerative}, for a spanning tree~$T \in \mathrm{SPT}(G)$ and a total order~$\prec$ on~$\widetilde{V}(G)$ we define
\[ \kappa(G,T,\prec) := \sum_{\substack{i,j \in \widetilde{V}(G), \\ i \leq_T j, \; j \prec i }} \omega_{G}(\{\mathrm{par}_T(i),j\}). \]

\begin{thm} \label{thm:gpfrsum}
We have
\[ \sum_{\alpha \in \mathrm{PF}(G)} q^{\mathrm{rsum}(\alpha)} = \sum_{T \in \mathrm{SPT}(G)}q^{\kappa(G,T,\prec)} \cdot \left(\prod_{e \in E(T)} [\omega_{G}(e)]_q\right)\]
where $\prec$ is any total order on $\widetilde{V}(G)$.
\end{thm}

Theorem~\ref{thm:gpfrsum} will follow immediately from the existence of a bijection
\[ \varphi^{\prec} \colon \mathrm{PF}(G) \to \left\{ (T,\ell)\colon \parbox{3in}{\begin{center} $T \in \mathrm{SPT}(n+1), \; \ell\colon E(T) \to \mathbb{N},$ \\ $\ell(e) \in \{0,1,\ldots,\omega_G(e)-1\} \textrm{ for all } e \in E(T)$ \end{center} } \right\}\]
satisfying $\mathrm{rsum}(\alpha) = \kappa(G,T,\prec) + \sum_{e \in E(T)}\ell(e)$ when $(T,\ell) = \varphi^{\prec}(\alpha)$. We establish the existence of such a bijection in Theorem~\ref{thm:multidfs} below. This bijection~$\varphi^{\prec}$ is a straightforward extension of the DFS-burning algorithm of~\cite{perkinson2013gparking} to multigraphs.

Let us recap the ``generic Dhar's burning algorithm.'' Dhar's burning algorithm takes an input a $G$-parking function and outputs a spanning tree of $G$. So suppose we are given some $\alpha = (\alpha_1,\ldots,\alpha_n) \in \mathrm{PF}(G)$. To initialize, we put $\alpha_i$ chips on vertex~$i$ for all~$1 \leq i \leq n$. The process is called the ``burning algorithm'' because we imagine that a fire is spreading throughout the graph: the chips can be seen as impediments that the fire has to burn through. At each step of the burning algorithm we have some collection of vertices that are \emph{burning} together with a rooted tree that spans the burning vertices. Initially only the root $0$ is burning. At each step of the algorithm we choose an edge connecting a burning vertex to one that is not yet burning. If the non-burning vertex has a positive number of chips on it, we remove one chip from that vertex and we \emph{burn} the edge we selected, removing it from consideration at later steps of the algorithm. If the non-burning vertex has no chips on it, that vertex starts to burn and we include the selected edge in the tree we are building up. Once all the vertices are burning we terminate and output the resulting spanning tree $T$. Importantly, the burning algorithm keeps track of the sum of its input: the sum of $\alpha$ is precisely the number of edges that were burnt. The reversed sum is the number of edges that were not burnt and are not in~$T$; we call these edges the \emph{surviving edges}.

To turn the ``generic'' burning algorithm into an actual, deterministic process we need to specify a procedure for choosing an edge between a burning and non-burning vertex at each step: let us call this choice the choice of which edge to \emph{burn along}. In the Cori-Le~Borgne~\cite{cori2003sand} variant of the burning algorithm we have some fixed total order on the edges and always choose to burn along the maximum edge according to this order. The surviving edges are precisely the edges that are externally active in the output tree~$T$; consequently, the Cori-Le~Borgne algorithm relates reversed sums of~$G$-parking functions to external activities of spanning trees of $G$. In the DFS-burning algorithm of Perkinson-Yang-Yu~\cite{perkinson2013gparking} we have some fixed total order~$\prec$ on the vertices and choose to burn along the edge connecting the burning vertices to the maximum non-burning vertex according to~$\prec$, but in a depth-first search fashion. In the case of a simple graph~$G$ (i.e.,~a graph with no multiple edges), the surviving edges are precisely those of the form $\{\mathrm{par}_T(i),j\}$ with~$i \leq_T j$ but~$j \prec i$ (where~$T$ is the output tree); consequently, the DFS-burning algorithm relates reversed sums of~$G$-parking functions to~$\kappa$-numbers of spanning trees of $G$. Incidentally, there are many other possible graph search procedures that one could use to specify which edge to burn along. Some of these, such as \emph{breadth-first search} and the  \emph{neighbor-first search} of Gessel and Sagan~\cite{gessel1996tutte} were explored by Kosti\'{c} and Yan~\cite{kostic2008multiparking}. A large family of graph search procedures were also investigated by Chebikin and Pylyavskyy in~\cite{chebikin2005family}. Each choice of graph search procedure relates reversed sum to some new spanning tree statistic.

The one wrinkle that can occur when applying the DFS-burning algorithm to multigraphs is that some of the surviving edges can be parallel to edges of the output tree~$T$, in which case these edges do not correspond to inversions. But all we have to do to correct for this is to give the output tree~$T$ an edge-labeling function $\ell\colon E(T) \to \mathbb{N}$ that records how many surviving edges were parallel to each $e \in E(T)$. Algorithm~\ref{alg:multidfs} gives pseudocode for the resulting multigraph DFS-burning algorithm. Algorithm~\ref{alg:multidfsinv} gives pseudocode for the inverse of this algorithm: the inverse algorithm uses a very similar depth-first search burning procedure but now instead of removing chips, it adds chips to vertices as edges are burnt.
  
 {\small
\begin{algorithm}
\caption{ {\bf Multigraph DFS-burning algorithm}.} \label{alg:multidfs}
\begin{algorithmic}[1]
  \Statex
  \Statex{\hspace{-0.5cm}\sc algorithm (with respect to multigraph $G$ and total order $\prec$ on $\widetilde{V}(G)$)}
  \Statex{\bf Input:} $\alpha = (\alpha_1,\ldots,\alpha_n) \in \mathbb{N}^{n}$
  \State $\texttt{burning\_vertices} := \{0\}$
  \State $\texttt{burnt\_edges} := \varnothing$
   \State $T := \textrm{ multigraph with $V(T) := \{0\}$ and $E(T) := \varnothing$}$
  \State $\ell := \textrm{ unique map $E(T) \to \mathbb{N}$}$
  \State $\texttt{head} := 0$
  \ForAll{$j \in V(G) - \texttt{burning\_vertices}$ in order from max to min according to $\prec$} \label{line:loop}
  	\ForAll{$e = \{\texttt{head},j\} \in E(G) - \texttt{burnt\_edges}$}
		\If{$\alpha_j \geq 1$}
			\State $\texttt{burnt\_edges} := \texttt{burnt\_edges} \cup \{e\}$
			\State $\alpha_j := \alpha_j - 1$
		\Else
			\State $\texttt{burning\_vertices} := \texttt{burning\_vertices} \cup \{j\}$
			\State $V(T) := V(T) \cup \{j\}$
			\State $E(T) := E(T) \cup \{e\}$
			\State $\ell(e) := \#\{f \in E(G) - (\texttt{burnt\_edges} \cup \{e\})\colon f = \{\texttt{head},j\}\}$
			\State $\texttt{head} := j$
			\State \textbf{break} out of current \textbf{for} loops and \Goto{line:loop}
		\EndIf
	\EndFor
  \EndFor
  \If{$\texttt{head} \neq 0$}
  	\State $\texttt{head} := \mathrm{par}_T(\texttt{head})$
	\State \Goto{line:loop}
  \EndIf
  \Statex{\bf Output:} {$(T,\ell)$}
  \end{algorithmic}
\end{algorithm}
}

 {\small
\begin{algorithm}
\caption{ {\bf Multigraph DFS-burning inverse algorithm}.} \label{alg:multidfsinv}
\begin{algorithmic}[1]
  \Statex
  \Statex{\hspace{-0.5cm}\sc algorithm (with respect to multigraph $G$ and total order $\prec$ on $\widetilde{V}(G)$)}
  \Statex{\bf Input:} $(T,\ell)$ with $T \in \mathrm{SPT}(G)$ and $\ell\colon E(T) \to \mathbb{N}$
  \State $\texttt{burning\_vertices} := \{0\}$
  \State $\texttt{burnt\_edges} := \varnothing$
   \State $\alpha = (\alpha_1,\ldots,\alpha_n) := (0,0,\ldots,0)$
  \State $\texttt{head} := 0$
  \ForAll{$j \in V(G) - \texttt{burning\_vertices}$ in order from max to min according to $\prec$} \label{line:loop2}
  	\ForAll{$e = \{\texttt{head},j\} \in E(G) - \texttt{burnt\_edges} $}
		\If{$e \in E(T)$ and $\ell(e) = \#\{f \in E(G) - (\texttt{burnt\_edges} \cup \{e\})\colon f = \{\texttt{head},j\}\}$}
			\State $\texttt{burning\_vertices} := \texttt{burning\_vertices} \cup \{j\}$
			\State $\texttt{head} := j$
			\State \textbf{break} out of current \textbf{for} loops and \Goto{line:loop2}	
		\Else
			\State $\texttt{burnt\_edges} := \texttt{burnt\_edges} \cup \{e\}$
			\State $\alpha_j := \alpha_j + 1$			
		\EndIf
	\EndFor
  \EndFor
  \If{$\texttt{head} \neq 0$}
  	\State $\texttt{head} := \mathrm{par}_T(\texttt{head})$
	\State \Goto{line:loop2}
  \EndIf
  \Statex{\bf Output:} {$\alpha$}
  \end{algorithmic}
\end{algorithm}
}

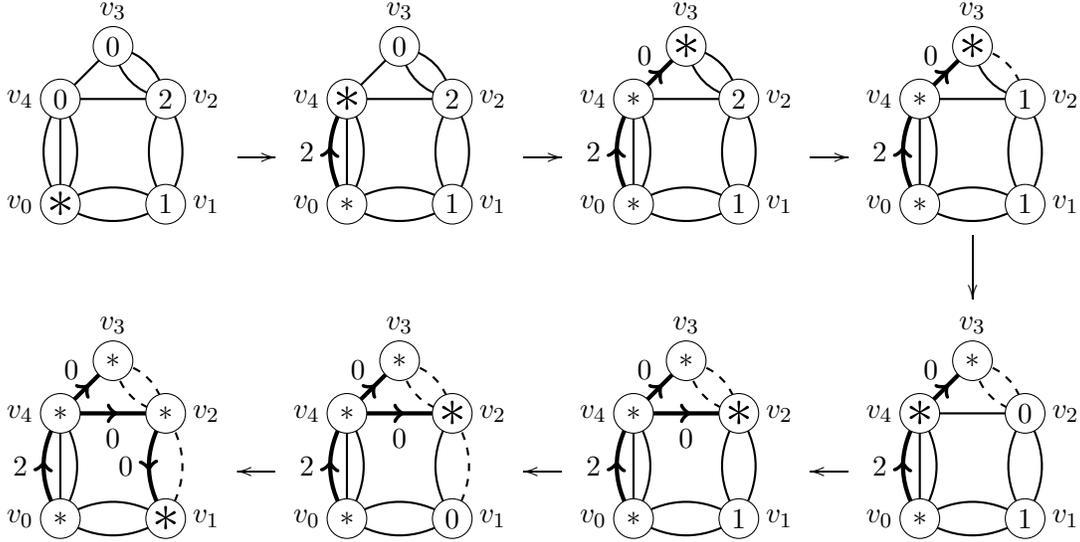
\begin{figure}
 \centerline{
 \xymatrix@C=3ex{
 \begin{tikzpicture}[scale=0.7]
	\SetFancyGraph
	\Vertex[LabelOut,Lpos=180, Ldist=.15cm,x=0,y=0,L={v_0}]{0}
	\Vertex[LabelOut,Lpos=0, Ldist=.15cm,x=2,y=0,L={v_1}]{1}
	\Vertex[LabelOut,Lpos=0, Ldist=.15cm,x=2,y=2,L={v_2}]{2}
	\Vertex[LabelOut,Lpos=90, Ldist=.15cm,x=1,y=3,L={v_3}]{3}
	\Vertex[LabelOut,Lpos=180, Ldist=.15cm,x=0,y=2,L={v_4}]{4}
	\Edges[style={thick,bend left}](0,1)
	\Edges[style={thick, bend right}](0,1)
	\Edges[style={thick, bend left}](1,2)
	\Edges[style={thick, bend right}](1,2)
	\Edges[style={thick, bend left}](2,3)
	\Edges[style={thick, bend right}](2,3)
	\Edges[style={thick}](2,4)
	\Edges[style={thick}](3,4)
	\Edges[style={thick, bend left}](0,4)
	\Edges[style={thick}](0,4)
	\Edges[style={thick, bend right}](0,4)
	\node[chips] at (0) {\huge{$*$}};
	\node[chips] at (1) {$1$};
	\node[chips] at (2) {$2$};
	\node[chips] at (3) {$0$};
	\node[chips] at (4) {$0$};
\end{tikzpicture} \ar@<5ex>[r] & \begin{tikzpicture}[scale=0.7,auto]
	\SetFancyGraph
	\Vertex[LabelOut,Lpos=180, Ldist=.15cm,x=0,y=0,L={v_0}]{0}
	\Vertex[LabelOut,Lpos=0, Ldist=.15cm,x=2,y=0,L={v_1}]{1}
	\Vertex[LabelOut,Lpos=0, Ldist=.15cm,x=2,y=2,L={v_2}]{2}
	\Vertex[LabelOut,Lpos=90, Ldist=.15cm,x=1,y=3,L={v_3}]{3}
	\Vertex[LabelOut,Lpos=180, Ldist=.15cm,x=0,y=2,L={v_4}]{4}
	\Edges[style={thick,bend left}](0,1)
	\Edges[style={thick, bend right}](0,1)
	\Edges[style={thick, bend left}](1,2)
	\Edges[style={thick, bend right}](1,2)
	\Edges[style={thick, bend left}](2,3)
	\Edges[style={thick, bend right}](2,3)
	\Edges[style={thick}](2,4)
	\Edges[style={thick}](3,4)
	\tikzstyle{LabelStyle}=[left=3pt,inner sep=0pt,outer sep=3pt]
	\Edges[style={ultra thick, bend left, ->--},label=$2$](0,4)
	\Edges[style={thick}](0,4)
	\Edges[style={thick, bend right}](0,4)
	\node[chips] at (0) {$*$};
	\node[chips] at (1) {$1$};
	\node[chips] at (2) {$2$};
	\node[chips] at (3) {$0$};
	\node[chips] at (4) {\huge{$*$}};
\end{tikzpicture} \ar@<5ex>[r] & \begin{tikzpicture}[scale=0.7]
	\SetFancyGraph
	\Vertex[LabelOut,Lpos=180, Ldist=.15cm,x=0,y=0,L={v_0}]{0}
	\Vertex[LabelOut,Lpos=0, Ldist=.15cm,x=2,y=0,L={v_1}]{1}
	\Vertex[LabelOut,Lpos=0, Ldist=.15cm,x=2,y=2,L={v_2}]{2}
	\Vertex[LabelOut,Lpos=90, Ldist=.15cm,x=1,y=3,L={v_3}]{3}
	\Vertex[LabelOut,Lpos=180, Ldist=.15cm,x=0,y=2,L={v_4}]{4}
	\Edges[style={thick,bend left}](0,1)
	\Edges[style={thick, bend right}](0,1)
	\Edges[style={thick, bend left}](1,2)
	\Edges[style={thick, bend right}](1,2)
	\Edges[style={thick, bend left}](2,3)
	\Edges[style={thick, bend right}](2,3)
	\Edges[style={thick}](2,4)
	\Edges[style={ultra thick, ->--},label=$0$](4,3)
	\tikzstyle{LabelStyle}=[left=3pt,inner sep=0pt,outer sep=3pt]
	\Edges[style={ultra thick, bend left, ->--},label=$2$](0,4)
	\Edges[style={thick}](0,4)
	\Edges[style={thick, bend right}](0,4)
	\node[chips] at (0) {$*$};
	\node[chips] at (1) {$1$};
	\node[chips] at (2) {$2$};
	\node[chips] at (3) {\huge{$*$}};
	\node[chips] at (4) {$*$};
\end{tikzpicture} \ar@<5ex>[r] & \begin{tikzpicture}[scale=0.7]
	\SetFancyGraph
	\Vertex[LabelOut,Lpos=180, Ldist=.15cm,x=0,y=0,L={v_0}]{0}
	\Vertex[LabelOut,Lpos=0, Ldist=.15cm,x=2,y=0,L={v_1}]{1}
	\Vertex[LabelOut,Lpos=0, Ldist=.15cm,x=2,y=2,L={v_2}]{2}
	\Vertex[LabelOut,Lpos=90, Ldist=.15cm,x=1,y=3,L={v_3}]{3}
	\Vertex[LabelOut,Lpos=180, Ldist=.15cm,x=0,y=2,L={v_4}]{4}
	\Edges[style={thick,bend left}](0,1)
	\Edges[style={thick, bend right}](0,1)
	\Edges[style={thick, bend left}](1,2)
	\Edges[style={thick, bend right}](1,2)
	\Edges[style={thick, bend left, dashed}](3,2)
	\Edges[style={thick, bend right}](3,2)
	\Edges[style={thick}](2,4)
	\Edges[style={ultra thick, ->--},label=$0$](4,3)
	\tikzstyle{LabelStyle}=[left=3pt,inner sep=0pt,outer sep=3pt]
	\Edges[style={ultra thick, bend left, ->--},label=$2$](0,4)
	\Edges[style={thick}](0,4)
	\Edges[style={thick, bend right}](0,4)
	\node[chips] at (0) {$*$};
	\node[chips] at (1) {$1$};
	\node[chips] at (2) {$1$};
	\node[chips] at (3) {\huge{$*$}};
	\node[chips] at (4) {$*$};
\end{tikzpicture} \ar[d]  \\
 \begin{tikzpicture}[scale=0.7]
	\SetFancyGraph
	\Vertex[LabelOut,Lpos=180, Ldist=.15cm,x=0,y=0,L={v_0}]{0}
	\Vertex[LabelOut,Lpos=0, Ldist=.15cm,x=2,y=0,L={v_1}]{1}
	\Vertex[LabelOut,Lpos=0, Ldist=.15cm,x=2,y=2,L={v_2}]{2}
	\Vertex[LabelOut,Lpos=90, Ldist=.15cm,x=1,y=3,L={v_3}]{3}
	\Vertex[LabelOut,Lpos=180, Ldist=.15cm,x=0,y=2,L={v_4}]{4}
	\tikzstyle{LabelStyle}=[left=3 pt,inner sep=0pt,outer sep=3pt]
	\Edges[style={ultra thick, bend right, ->--},label=$0$](2,1)
	\Edges[style={thick,bend left}](0,1)
	\Edges[style={thick, bend right}](0,1)
	\Edges[style={thick, bend left, dashed}](2,1)
	\Edges[style={thick, bend left, dashed}](3,2)
	\Edges[style={thick, bend right, dashed}](3,2)
	\tikzstyle{LabelStyle}=[auto=3pt,inner sep=0pt,outer sep=3pt]
	\Edges[style={ultra thick, ->--},label=$0$](4,3)
	\tikzstyle{LabelStyle}=[below =3pt,inner sep=0pt,outer sep=3pt]
	\Edges[style={ultra thick, ->--},label=$0$](4,2)
	\tikzstyle{LabelStyle}=[left=3pt,inner sep=0pt,outer sep=3pt]
	\Edges[style={ultra thick, bend left, ->--},label=$2$](0,4)
	\Edges[style={thick}](0,4)
	\Edges[style={thick, bend right}](0,4)
	\node[chips] at (0) {$*$};
	\node[chips] at (1) {\huge{$*$}};
	\node[chips] at (2) {$*$};
	\node[chips] at (3) {$*$};
	\node[chips] at (4) {$*$};
\end{tikzpicture}  & \ar@<-5ex>[l] \begin{tikzpicture}[scale=0.7]
	\SetFancyGraph
	\Vertex[LabelOut,Lpos=180, Ldist=.15cm,x=0,y=0,L={v_0}]{0}
	\Vertex[LabelOut,Lpos=0, Ldist=.15cm,x=2,y=0,L={v_1}]{1}
	\Vertex[LabelOut,Lpos=0, Ldist=.15cm,x=2,y=2,L={v_2}]{2}
	\Vertex[LabelOut,Lpos=90, Ldist=.15cm,x=1,y=3,L={v_3}]{3}
	\Vertex[LabelOut,Lpos=180, Ldist=.15cm,x=0,y=2,L={v_4}]{4}
	\Edges[style={thick,bend left}](0,1)
	\Edges[style={thick, bend right}](0,1)
	\Edges[style={thick, bend left, dashed}](2,1)
	\Edges[style={thick, bend right}](2,1)
	\Edges[style={thick, bend left, dashed}](3,2)
	\Edges[style={thick, bend right, dashed}](3,2)
	\Edges[style={ultra thick, ->--},label=$0$](4,3)
	\tikzstyle{LabelStyle}=[below=3pt,inner sep=0pt,outer sep=3pt]
	\Edges[style={ultra thick, ->--},label=$0$](4,2)
	\tikzstyle{LabelStyle}=[left=3pt,inner sep=0pt,outer sep=3pt]
	\Edges[style={ultra thick, bend left, ->--},label=$2$](0,4)
	\Edges[style={thick}](0,4)
	\Edges[style={thick, bend right}](0,4)
	\node[chips] at (0) {$*$};
	\node[chips] at (1) {$0$};
	\node[chips] at (2) {\huge{$*$}};
	\node[chips] at (3) {$*$};
	\node[chips] at (4) {$*$};
\end{tikzpicture} &  \ar@<-5ex>[l] \begin{tikzpicture}[scale=0.7]
	\SetFancyGraph
	\Vertex[LabelOut,Lpos=180, Ldist=.15cm,x=0,y=0,L={v_0}]{0}
	\Vertex[LabelOut,Lpos=0, Ldist=.15cm,x=2,y=0,L={v_1}]{1}
	\Vertex[LabelOut,Lpos=0, Ldist=.15cm,x=2,y=2,L={v_2}]{2}
	\Vertex[LabelOut,Lpos=90, Ldist=.15cm,x=1,y=3,L={v_3}]{3}
	\Vertex[LabelOut,Lpos=180, Ldist=.15cm,x=0,y=2,L={v_4}]{4}
	\Edges[style={thick,bend left}](0,1)
	\Edges[style={thick, bend right}](0,1)
	\Edges[style={thick, bend left}](1,2)
	\Edges[style={thick, bend right}](1,2)
	\Edges[style={thick, bend left, dashed}](3,2)
	\Edges[style={thick, bend right, dashed}](3,2)
	\Edges[style={ultra thick, ->--},label=$0$](4,3)
	\tikzstyle{LabelStyle}=[below=3pt,inner sep=0pt,outer sep=3pt]
	\Edges[style={ultra thick, ->--},label=$0$](4,2)
	\tikzstyle{LabelStyle}=[left=3pt,inner sep=0pt,outer sep=3pt]
	\Edges[style={ultra thick, bend left, ->--},label=$2$](0,4)
	\Edges[style={thick}](0,4)
	\Edges[style={thick, bend right}](0,4)
	\node[chips] at (0) {$*$};
	\node[chips] at (1) {$1$};
	\node[chips] at (2) {\huge{$*$}};
	\node[chips] at (3) {$*$};
	\node[chips] at (4) {$*$};
\end{tikzpicture} & \ar@<-5ex>[l] \begin{tikzpicture}[scale=0.7]
	\SetFancyGraph
	\Vertex[LabelOut,Lpos=180, Ldist=.15cm,x=0,y=0,L={v_0}]{0}
	\Vertex[LabelOut,Lpos=0, Ldist=.15cm,x=2,y=0,L={v_1}]{1}
	\Vertex[LabelOut,Lpos=0, Ldist=.15cm,x=2,y=2,L={v_2}]{2}
	\Vertex[LabelOut,Lpos=90, Ldist=.15cm,x=1,y=3,L={v_3}]{3}
	\Vertex[LabelOut,Lpos=180, Ldist=.15cm,x=0,y=2,L={v_4}]{4}
	\Edges[style={thick,bend left}](0,1)
	\Edges[style={thick, bend right}](0,1)
	\Edges[style={thick, bend left}](1,2)
	\Edges[style={thick, bend right}](1,2)
	\Edges[style={thick, bend left, dashed}](3,2)
	\Edges[style={thick, bend right, dashed}](3,2)
	\Edges[style={thick}](2,4)
	\Edges[style={ultra thick, ->--},label=$0$](4,3)
	\tikzstyle{LabelStyle}=[left=3pt,inner sep=0pt,outer sep=3pt]
	\Edges[style={ultra thick, bend left, ->--},label=$2$](0,4)
	\Edges[style={thick}](0,4)
	\Edges[style={thick, bend right}](0,4)
	\node[chips] at (0) {$*$};
	\node[chips] at (1) {$1$};
	\node[chips] at (2) {$0$};
	\node[chips] at (3) {$*$};
	\node[chips] at (4) {\huge{$*$}};
\end{tikzpicture} } }
\caption{Example~\ref{ex:multidfs}: a run of the multigraph DFS-burning algorithm. The algorithm starts in the upper-left and ends in the lower-left. Edges of the graph are drawn with multiplicity for clarity. At each step, the burning vertices are marked by an asterisk and the non-burning vertices are marked by the number of chips remaining on them. The head of the depth-first search has a larger asterisk. The edges in the tree spanning the burning vertices are bold and directed away from the root; their labels are displayed beside them. The burnt edges are dashed.} \label{fig:multidfs}
 \end{figure}
 
 \begin{example} \label{ex:multidfs}
Let $G$ be the multigraph below, where for clarity we have labeled each edge $e \in E(G)$ by $\omega_G(e)$ and each vertex $i \in V(G)$ by $v_i$:
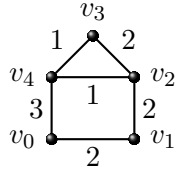
\begin{center}
\begin{tikzpicture}[scale=0.55]
	\SetFancyGraph
	\Vertex[LabelOut,Lpos=180, Ldist=0cm,x=0,y=0.5,L={v_0}]{0}
	\Vertex[LabelOut,Lpos=0, Ldist=0cm,x=2,y=0.5,L={v_1}]{1}
	\Vertex[LabelOut,Lpos=0, Ldist=0cm,x=2,y=2,L={v_2}]{2}
	\Vertex[LabelOut,Lpos=90, Ldist=0cm,x=1,y=3,L={v_3}]{3}
	\Vertex[LabelOut,Lpos=180, Ldist=0cm,x=0,y=2,L={v_4}]{4}
	\Edges[style={thick},label={2}](1,0)
	\Edges[style={thick},label={2}](2,1)
	\Edges[style={thick},label={2}](3,2)
	\Edges[style={thick},label={1}](2,4)
	\Edges[style={thick},label={1}](4,3)
	\Edges[style={thick},label={3}](0,4)
\end{tikzpicture}
\end{center}
Let $\alpha := (1,2,0,0) \in \mathrm{PF}(G)$ and let $\prec$ be given by $ 1 \prec 2 \prec 3 \prec 4$. Figure~\ref{fig:multidfs} depicts a run on the multigraph DFS-burning algorithm on $\alpha$ with respect to the vertex order~$\prec$. The result of the DFS-burning algorithm is that $(T,\ell) := \varphi^{\prec}(\alpha)$ is
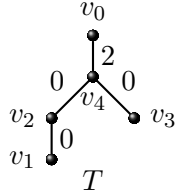
\begin{center}
\begin{tikzpicture}[scale=0.55]
	\SetFancyGraph
	\Vertex[LabelOut,Lpos=90, Ldist=0cm,x=0,y=0,L={v_0}]{0}
	\Vertex[LabelOut,Lpos=180, Ldist=0cm,x=-1,y=-3,L={v_1}]{1}
	\Vertex[LabelOut,Lpos=180, Ldist=0cm,x=-1,y=-2,L={v_2}]{2}
	\Vertex[LabelOut,Lpos=0, Ldist=0cm,x=1,y=-2,L={v_3}]{3}
	\Vertex[LabelOut,Lpos=270, Ldist=0cm,x=0,y=-1,L={v_4}]{4}
	\Edges[style={thick},label=$2$](0,4)
	\Edges[style={thick},label=$0$](2,4)
	\Edges[style={thick},label=$0$](4,3)
	\Edges[style={thick},label=$0$](2,1)
	\node at (0,-3.5) {$T$};
\end{tikzpicture}
\end{center}
where now we label each $e \in E(T)$ by $\ell(e)$. We can verify that $\mathrm{rsum}(\alpha)$ is 
\[ 4 = \omega_G(\{0,2\}) + \omega_G(\{0,3\}) + \omega_G(\{4,1\}) + \omega_G(\{0,1\}) + 2 = \kappa(G,T,\prec) + \hspace{-0.2cm} \sum_{e \in E(T)} \hspace{-0.2cm} \ell(e).\]
Observe also that $4$ is the number of surviving edges, that is, the number of solid edges not in $T$ in the lower-left of Figure~\ref{fig:multidfs}.
\end{example}

\begin{thm} \label{thm:multidfs}
Let $(T,\ell)$ be the output of an application of Algorithm~\ref{alg:multidfs} to $\alpha \in \mathbb{N}^{n-1}$ with respect to any total order $\prec$ on $\widetilde{V}(G)$. If~$V(T) = V(G)$ then $\alpha \in \mathrm{PF}(G)$. And if~$V(T) \neq V(G)$, then the set $U := V(G) - V(T)$ certifies that $\alpha$ is not a $G$-parking function in the sense that for all $j \in U$ we have~$a_j \geq d^{G}_U(j)$.

For a fixed total order $\prec$ on $\widetilde{V}(G)$, associating to each~$\alpha \in \mathrm{PF}(G)$ the edge-labeled spanning tree $(T,\ell)$ produced by Algorithm~\ref{alg:multidfs} on input $\alpha$ yields a bijection
\[ \varphi^{\prec} \colon \mathrm{PF}(G) \to \left\{ (T,\ell)\colon \parbox{3in}{\begin{center} $T \in \mathrm{SPT}(G), \; \ell\colon E(T) \to \mathbb{N},$ \\ $\ell(e) \in \{0,1,\ldots,\omega_G(e)-1\} \textrm{ for all } e \in E(T)$ \end{center} } \right\}.\]
This bijection satisfies $\mathrm{rsum}(\alpha) = \kappa(G,T,\prec) + \sum_{e \in E(T)}\ell(e)$ when $(T,\ell) = \varphi^{\prec}(\alpha)$. The inverse of the bijection is given by Algorithm~\ref{alg:multidfsinv}.
\end{thm}
\begin{proof}
The proof of the correctness for this algorithm is so similar to that of the original algorithm in~\cite{perkinson2013gparking} that we will not repeat it here.
\end{proof}

\begin{remark} \label{rem:yan}
When $\mathbf{x} = (a,\overbrace{b,b,\ldots,b}^{n-1})$ we have $\mathrm{PF}(\mathbf{x}) = \mathrm{PF}(K_{n+1}^{a,b})$ by Theorem~\ref{thm:gpfandxpf}. Thus Theorem~\ref{thm:gpfrsum} recovers a result of Yan~\cite{yan2001generalized} expressing the reversed sum enumerator for these $\mathbf{x}$-parking functions in terms of inversions in ``rooted $b$-forests.''
\end{remark}

\section{The vector DFS-burning algorithm} \label{sec:vecdfs}

In this section we prove our main result, Theorem~\ref{thm:main}. Fix $\mathbf{x} = (x_1,\ldots,x_n) \in \mathbb{N}^n$. Assume~$x_1 \geq 1$ as otherwise Theorem~\ref{thm:main} holds trivially since~$\mathrm{PF}(\mathbf{x}) = \varnothing$. Define a connected multigraph $G_{\mathbf{x}}$ associated to $\mathbf{x}$: its vertex set is $V(G_{\mathbf{x}}) := \{0,1,\ldots,n\}$ and its edge-weight function is given by~$\omega_{G_{\mathbf{x}}}(\{i,j\}) := x_{\mathrm{min}(i,j)+1}$ for all $0\leq i,j \leq n$. As always, $0$ is the root of $G_{\mathbf{x}}$. Note that~$\#E(G_{\mathbf{x}}) - n =  \left( \sum_{i=1}^{n} (n+1-i) x_{i} \right) - n $ so the notions of reversed sum for $G_{\mathbf{x}}$-parking functions and $\mathbf{x}$-parking functions agree. Theorem~\ref{thm:main} will follow immediately from the existence of a bijection 
\[ \psi \colon \mathrm{PF}(\mathbf{x}) \to \left\{ (T,\ell,\prec)\colon \parbox{3in}{\begin{center} $T \in \mathrm{RPT}(n+1), \; \prec \in \mathrm{AVO}(T), \; \ell\colon E(T) \to \mathbb{N},$ \\ $\ell(e) \in \{0,1,\ldots,\omega_{G_{\mathbf{x}}}(e)-1\} \textrm{ for all } e \in E(T)$ \end{center} } \right\}\]
satisfying $\mathrm{rsum}(\alpha) = \kappa(G_{\mathbf{x}},T,\prec) + \sum_{e \in E(T)}\ell(e)$ when $(T,\ell,\prec) = \psi(\alpha)$. We establish the existence of such a bijection in Theorem~\ref{thm:vecdfs} below. This bijection is a far more substantial variation of the DFS-burning algorithm than the extension to multigraphs in the last section. We want to ``symmetrize'' the DFS-burning algorithm. The trick is that we will build up our graph $G_{\mathbf{x}}$ as we burn through it. We start with a graph on~$\{0,1,\ldots,n\}$ which has no edges. Then~$0$ starts to burn. When a vertex $j$ becomes the $i$th vertex to start burning we add $x_{i}$ edges between $j$ and all vertices that are not yet burning. An auxiliary permutation $\sigma$ in the algorithm records the order in which vertices started to burn. The permutation~$\sigma$ determines the admissible vertex order~$\prec$ we return together with the rooted plane tree $T$ and its edge-labeling function~$\ell$. When we terminate the algorithm the graph we built up will always be a copy of $G_{\mathbf{x}}$ once we relabel the vertices by~$i \mapsto \sigma^{-1}(i)$. The point is that this vector DFS-burning algorithm simulates a run of the multigraph DFS-burning algorithm on this relabeled multigraph~$G_{\mathbf{x}}$ with respect to the order~$\prec$ we obtain as output. Therefore, when analyzing this new procedure we can invoke all the nice results about the behavior of the multigraph DFS-burning algorithm detailed in Theorem~\ref{thm:multidfs}: that it respects the reversed sum statistic, that it is invertible, et cetera. Algorithm~\ref{alg:vecdfs} gives pseudocode for the vector DFS-burning algorithm. Algorithm~\ref{alg:vecdfsinv} gives pseudocode for the inverse of this algorithm, which just applies the multigraph DFS-burning inverse algorithm.

{\small
\begin{algorithm}
\caption{ {\bf Vector DFS-burning algorithm}.} \label{alg:vecdfs}
\begin{algorithmic}[1]
  \Statex
  \Statex{\hspace{-0.5cm}\sc algorithm (with respect to $\mathbf{x} = (x_1,\ldots,x_{n}) \in \mathbb{N}^n$)}
  \Statex{\bf Input:} $\alpha = (\alpha_1,\ldots,\alpha_n) \in \mathbb{N}^n$
  \State $G := \textrm{multigraph with $V(G) := \{0,1,\ldots,n\}$ and $E(G) := \varnothing$}$
  \State $\texttt{burning\_vertices} := \{0\}$
  \State $\texttt{burnt\_edges} := \varnothing$
  \State $\texttt{counter} := 0$
  \State $\sigma := \textrm{ unique map $\{0\} \to \{0\}$}$  
  \ForAll {$k \in V(G) - \texttt{burning\_vertices}$}:
  	\State Add $x_1$ copies of the edge $\{0,k\}$ to $E(G)$
  \EndFor
  \State $T := \textrm{ multigraph with $V(T) := \{0\}$ and $E(T) := \varnothing$}$
  \State $\ell := \textrm{ unique map $E(T) \to \mathbb{N}$}$
  \State $\texttt{head} := 0$
  \ForAll{$j \in V(G) - \texttt{burning\_vertices}$ in order from max to min $j$} \label{line:loop3}
  	\ForAll{$e = \{\texttt{head},j\} \in E(G) - \texttt{burnt\_edges}$}
		\If{$\alpha_j \geq 1$}
			\State $\texttt{burnt\_edges} := \texttt{burnt\_edges} \cup \{e\}$
			\State $\alpha_j := \alpha_j - 1$
		\Else
			\State $\texttt{burning\_vertices} := \texttt{burning\_vertices} \cup \{j\}$
			\State $\texttt{counter} := \texttt{counter} + 1$
			\State $\sigma(\texttt{counter}) := j$
			\ForAll {$k \in V(G) - \texttt{burning\_vertices}$}:
				 \State Add $x_{\texttt{counter}+1}$ copies of the edge $\{j,k\}$ to $E(G)$
			\EndFor
			\State $V(T) := V(T) \cup \{j\}$
			\State $E(T) := E(T) \cup \{\texttt{head},j\}$
			\State $\ell(\{\texttt{head},j\}) := \#\{f \in E(G) - (\texttt{burnt\_edges} \cup \{e\})\colon f = \{\texttt{head},j\}\}$
  			\State $\texttt{head} := j$
			\State \textbf{break} out of current \textbf{for} loops and \Goto{line:loop3}
		\EndIf
	\EndFor
  \EndFor
  \If{$\texttt{head} \neq 0$}
  	\State $\texttt{head} := \mathrm{par}_T(\texttt{head})$
	\State \Goto{line:loop3}
  \EndIf
  \State relabel the vertices of $T$ by $i \mapsto \sigma^{-1}(i)$ for all $i \in V(T)$
  \State $\prec \; := \textrm{ the unique total order on $\widetilde{V}(T)$ with $i \prec j$ if and only if $\sigma(i) < \sigma(j)$}$
  \Statex{\bf Output:} {$(T,\ell,\prec)$}
  \end{algorithmic}
\end{algorithm}
}

{\small
\begin{algorithm}
\caption{ {\bf Vector DFS-burning inverse algorithm}.} \label{alg:vecdfsinv}
\begin{algorithmic}[1]
  \Statex
  \Statex{\hspace{-0.5cm}\sc algorithm (with respect to $\mathbf{x} = (x_1,\ldots,x_{n}) \in \mathbb{N}^n$)}
  \Statex{\bf Input:} $(T,\ell,\prec)$ with $T \in \mathrm{RPT}(n+1)$, $\ell\colon E(T) \to \mathbb{N}$, and $\prec \in \mathrm{AVO}(T)$
  \State $\beta = (\beta_1,\ldots,\beta_n) := \textrm{ result of applying Algorithm~\ref{alg:multidfsinv} to $(T,\ell)$ with respect to $G_{\mathbf{x}}$ and $\prec$}$
  \State $\sigma := \textrm{ the unique permutation in $\mathfrak{S}_n$ with $\sigma(1) \prec \sigma(2) \prec \cdots \prec \sigma(n)$}$
  \State $\alpha := (\beta_{\sigma(1)},\beta_{\sigma(2)},\ldots,\beta_{\sigma(n)})$
  \Statex{\bf Output:} {$\alpha$}
  \end{algorithmic}
\end{algorithm}
}

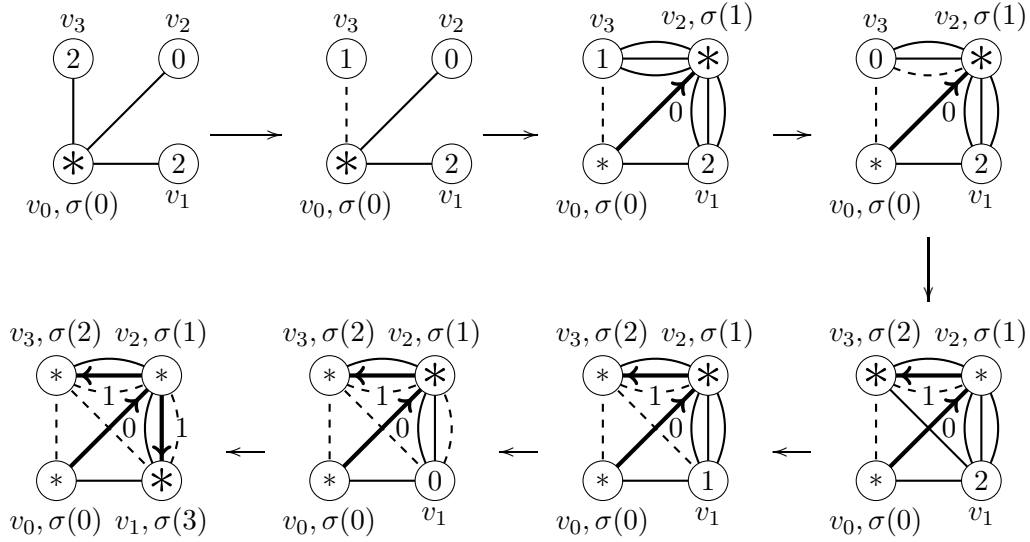
\begin{figure}
\centerline{
\xymatrix@C=3ex{
 \begin{tikzpicture}[scale=0.7]
	\SetFancyGraph
	\Vertex[LabelOut,Lpos=270, Ldist=.15cm,x=0,y=0,L={v_0,\sigma(0)}]{0}
	\Vertex[LabelOut,Lpos=270, Ldist=.15cm,x=2,y=0,L={v_1}]{1}
	\Vertex[LabelOut,Lpos=90, Ldist=.15cm,x=2,y=2,L={v_2}]{2}
	\Vertex[LabelOut,Lpos=90, Ldist=.15cm,x=0,y=2,L={v_3}]{3}
	\Edges[style={thick}](0,1)
	\Edges[style={thick}](0,2)
	\Edges[style={thick}](0,3)
	\node[chips] at (0) {\huge{$*$}};
	\node[chips] at (1) {$2$};
	\node[chips] at (2) {$0$};
	\node[chips] at (3) {$2$};
\end{tikzpicture} \ar@<7ex>[r] &  \begin{tikzpicture}[scale=0.7]
	\SetFancyGraph
	\Vertex[LabelOut,Lpos=270, Ldist=.15cm,x=0,y=0,L={v_0,\sigma(0)}]{0}
	\Vertex[LabelOut,Lpos=270, Ldist=.15cm,x=2,y=0,L={v_1}]{1}
	\Vertex[LabelOut,Lpos=90, Ldist=.15cm,x=2,y=2,L={v_2}]{2}
	\Vertex[LabelOut,Lpos=90, Ldist=.15cm,x=0,y=2,L={v_3}]{3}
	\Edges[style={thick}](0,1)
	\Edges[style={thick}](0,2)
	\Edges[style={thick,dashed}](0,3)
	\node[chips] at (0) {\huge{$*$}};
	\node[chips] at (1) {$2$};
	\node[chips] at (2) {$0$};
	\node[chips] at (3) {$1$};
\end{tikzpicture} \ar@<7ex>[r] & \begin{tikzpicture}[scale=0.7]
	\SetFancyGraph
	\Vertex[LabelOut,Lpos=270, Ldist=.15cm,x=0,y=0,L={v_0,\sigma(0)}]{0}
	\Vertex[LabelOut,Lpos=270, Ldist=.15cm,x=2,y=0,L={v_1}]{1}
	\Vertex[LabelOut,Lpos=90, Ldist=.15cm,x=2,y=2,L={v_2,\sigma(1)}]{2}
	\Vertex[LabelOut,Lpos=90, Ldist=.15cm,x=0,y=2,L={v_3}]{3}
	\Edges[style={thick}](0,1)
	\tikzstyle{LabelStyle}=[right=2pt,inner sep=0pt,outer sep=3pt]
	\Edges[style={ultra thick, ->-},label=$0$](0,2)
	\Edges[style={thick,dashed}](0,3)
	\Edges[style={thick, bend left=30}](2,1)
	\Edges[style={thick}](2,1)
	\Edges[style={thick, bend right=30}](2,1)
	\Edges[style={thick, bend left=30}](2,3)
	\Edges[style={thick}](2,3)
	\Edges[style={thick, bend right=30}](2,3)
	\node[chips] at (0) {$*$};
	\node[chips] at (1) {$2$};
	\node[chips] at (2) {\huge{$*$}};
	\node[chips] at (3) {$1$};
\end{tikzpicture} \ar@<7ex>[r] & \begin{tikzpicture}[scale=0.7]
	\SetFancyGraph
	\Vertex[LabelOut,Lpos=270, Ldist=.15cm,x=0,y=0,L={v_0,\sigma(0)}]{0}
	\Vertex[LabelOut,Lpos=270, Ldist=.15cm,x=2,y=0,L={v_1}]{1}
	\Vertex[LabelOut,Lpos=90, Ldist=.15cm,x=2,y=2,L={v_2,\sigma(1)}]{2}
	\Vertex[LabelOut,Lpos=90, Ldist=.15cm,x=0,y=2,L={v_3}]{3}
	\Edges[style={thick}](0,1)
	\tikzstyle{LabelStyle}=[right=2pt,inner sep=0pt,outer sep=3pt]
	\Edges[style={ultra thick, ->-},label=$0$](0,2)
	\Edges[style={thick,dashed}](0,3)
	\Edges[style={thick, bend left=30}](2,1)
	\Edges[style={thick}](2,1)
	\Edges[style={thick, bend right=30}](2,1)
	\Edges[style={thick, bend left=30, dashed}](2,3)
	\Edges[style={thick}](2,3)
	\Edges[style={thick, bend right=30}](2,3)
	\node[chips] at (0) {$*$};
	\node[chips] at (1) {$2$};
	\node[chips] at (2) {\huge{$*$}};
	\node[chips] at (3) {$0$};
\end{tikzpicture} \ar[d] \\
\begin{tikzpicture}[scale=0.7]
	\SetFancyGraph
	\Vertex[LabelOut,Lpos=270, Ldist=.15cm,x=0,y=0,L={v_0,\sigma(0)}]{0}
	\Vertex[LabelOut,Lpos=270, Ldist=.15cm,x=2,y=0,L={v_1,\sigma(3)}]{1}
	\Vertex[LabelOut,Lpos=90, Ldist=.15cm,x=2,y=2,L={v_2,\sigma(1)}]{2}
	\Vertex[LabelOut,Lpos=90, Ldist=.15cm,x=0,y=2,L={v_3,\sigma(2)}]{3}
	\Edges[style={thick}](0,1)
	\tikzstyle{LabelStyle}=[right=2pt,inner sep=0pt,outer sep=3pt]
	\Edges[style={ultra thick, ->-},label=$0$](0,2)
	\Edges[style={thick,dashed}](0,3)
	\Edges[style={thick, bend left=30,dashed}](2,1)
	\tikzstyle{LabelStyle}=[right=2pt,inner sep=0pt,outer sep=3pt]
	\Edges[style={ultra thick, ->-},label=$1$](2,1)
	\Edges[style={thick, bend right=30}](2,1)
	\Edges[style={thick, bend left=30, dashed}](2,3)
	\tikzstyle{LabelStyle}=[below=1pt,inner sep=0pt,outer sep=3pt]
	\Edges[style={ultra thick, ->-},label=$1$](2,3)
	\Edges[style={thick, bend right=30}](2,3)
	\Edges[style={thick,dashed}](3,1)
	\node[chips] at (0) {$*$};
	\node[chips] at (1) {\huge{$*$}};
	\node[chips] at (2) {$*$};
	\node[chips] at (3) {$*$};
\end{tikzpicture} & \ar@<-7ex>[l] \begin{tikzpicture}[scale=0.7]
	\SetFancyGraph
	\Vertex[LabelOut,Lpos=270, Ldist=.15cm,x=0,y=0,L={v_0,\sigma(0)}]{0}
	\Vertex[LabelOut,Lpos=270, Ldist=.15cm,x=2,y=0,L={v_1}]{1}
	\Vertex[LabelOut,Lpos=90, Ldist=.15cm,x=2,y=2,L={v_2,\sigma(1)}]{2}
	\Vertex[LabelOut,Lpos=90, Ldist=.15cm,x=0,y=2,L={v_3,\sigma(2)}]{3}
	\Edges[style={thick}](0,1)
	\tikzstyle{LabelStyle}=[right=2pt,inner sep=0pt,outer sep=3pt]
	\Edges[style={ultra thick, ->-},label=$0$](0,2)
	\Edges[style={thick,dashed}](0,3)
	\Edges[style={thick, bend left=30,dashed}](2,1)
	\Edges[style={thick}](2,1)
	\Edges[style={thick, bend right=30}](2,1)
	\Edges[style={thick, bend left=30, dashed}](2,3)
	\tikzstyle{LabelStyle}=[below=1pt,inner sep=0pt,outer sep=3pt]
	\Edges[style={ultra thick, ->-},label=$1$](2,3)
	\Edges[style={thick, bend right=30}](2,3)
	\Edges[style={thick,dashed}](3,1)
	\node[chips] at (0) {$*$};
	\node[chips] at (1) {$0$};
	\node[chips] at (2) {\huge{$*$}};
	\node[chips] at (3) {$*$};
\end{tikzpicture} & \ar@<-7ex>[l] \begin{tikzpicture}[scale=0.7]
	\SetFancyGraph
	\Vertex[LabelOut,Lpos=270, Ldist=.15cm,x=0,y=0,L={v_0,\sigma(0)}]{0}
	\Vertex[LabelOut,Lpos=270, Ldist=.15cm,x=2,y=0,L={v_1}]{1}
	\Vertex[LabelOut,Lpos=90, Ldist=.15cm,x=2,y=2,L={v_2,\sigma(1)}]{2}
	\Vertex[LabelOut,Lpos=90, Ldist=.15cm,x=0,y=2,L={v_3,\sigma(2)}]{3}
	\Edges[style={thick}](0,1)
	\tikzstyle{LabelStyle}=[right=2pt,inner sep=0pt,outer sep=3pt]
	\Edges[style={ultra thick, ->-},label=$0$](0,2)
	\Edges[style={thick,dashed}](0,3)
	\Edges[style={thick, bend left=30}](2,1)
	\Edges[style={thick}](2,1)
	\Edges[style={thick, bend right=30}](2,1)
	\Edges[style={thick, bend left=30, dashed}](2,3)
	\tikzstyle{LabelStyle}=[below=1pt,inner sep=0pt,outer sep=3pt]
	\Edges[style={ultra thick, ->-},label=$1$](2,3)
	\Edges[style={thick, bend right=30}](2,3)
	\Edges[style={thick,dashed}](3,1)
	\node[chips] at (0) {$*$};
	\node[chips] at (1) {$1$};
	\node[chips] at (2) {\huge{$*$}};
	\node[chips] at (3) {$*$};
\end{tikzpicture} & \ar@<-7ex>[l] \begin{tikzpicture}[scale=0.7]
	\SetFancyGraph
	\Vertex[LabelOut,Lpos=270, Ldist=.15cm,x=0,y=0,L={v_0,\sigma(0)}]{0}
	\Vertex[LabelOut,Lpos=270, Ldist=.15cm,x=2,y=0,L={v_1}]{1}
	\Vertex[LabelOut,Lpos=90, Ldist=.15cm,x=2,y=2,L={v_2,\sigma(1)}]{2}
	\Vertex[LabelOut,Lpos=90, Ldist=.15cm,x=0,y=2,L={v_3,\sigma(2)}]{3}
	\Edges[style={thick}](0,1)
	\tikzstyle{LabelStyle}=[right=2pt,inner sep=0pt,outer sep=3pt]
	\Edges[style={ultra thick, ->-},label=$0$](0,2)
	\Edges[style={thick,dashed}](0,3)
	\Edges[style={thick, bend left=30}](2,1)
	\Edges[style={thick}](2,1)
	\Edges[style={thick, bend right=30}](2,1)
	\Edges[style={thick, bend left=30, dashed}](2,3)
	\tikzstyle{LabelStyle}=[below=1pt,inner sep=0pt,outer sep=3pt]
	\Edges[style={ultra thick, ->-},label=$1$](2,3)
	\Edges[style={thick, bend right=30}](2,3)
	\Edges[style={thick}](3,1)
	\node[chips] at (0) {$*$};
	\node[chips] at (1) {$2$};
	\node[chips] at (2) {$*$};
	\node[chips] at (3) {\huge{$*$}};
\end{tikzpicture} } }

\caption{Example~\ref{ex:vecdfs}: a run of the vector DFS-burning algorithm. The algorithm starts in the upper-left and ends in the lower-left.} \label{fig:vecdfs}
\end{figure}

\medskip

 \begin{example} \label{ex:vecdfs}
 Set $\mathbf{x} := (1,3,1)$ and take $\alpha := (2,0,2) \in \mathrm{PF}(\mathbf{x})$. Figure~\ref{fig:vecdfs} depicts a run of the vector DFS-burning algorithm on $\alpha$. As in Example~\ref{ex:multidfs}, for clarity we label vertex $i \in V(G)$ by $v_i$. But note that the output tree $T$ will have its vertices relabeled by~$i \mapsto \sigma^{-1}(i)$ with respect to the depicted~$v_i$ labeling in Figure~\ref{fig:vecdfs}. Specifically, the result of the DFS-burning algorithm is that $(T,\ell,\prec) := \psi(\alpha)$ is
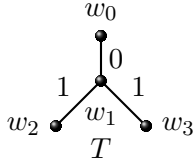
\begin{center}
\begin{tikzpicture}[scale=0.6]
	\SetFancyGraph
	\Vertex[LabelOut,Lpos=90, Ldist=0cm,x=0,y=0,L={w_0}]{0}
	\Vertex[LabelOut,Lpos=270, Ldist=.15cm,x=0,y=-1,L={w_1}]{1}
	\Vertex[LabelOut,Lpos=180, Ldist=0cm,x=-1,y=-2,L={w_2}]{2}
	\Vertex[LabelOut,Lpos=0, Ldist=0cm,x=1,y=-2,L={w_3}]{3}
	\Edges[style={thick},label=$0$](0,1)
	\Edges[style={thick},label=$1$](2,1)
	\Edges[style={thick},label=$1$](1,3)
	\node at (0,-2.5) {$T$};
\end{tikzpicture}
\end{center}
where we label each $i \in V(T)$ by $w_i$, each $e \in E(T)$ by $\ell(e)$, and $3 \prec 1 \prec 2$. We can verify that~$\mathrm{rsum}(\alpha) = 3 = \omega_{G_{\mathbf{x}}}(\{0,3\}) + 2 = \kappa(G,T,\prec) + \sum_{e \in E(T)} \ell(e)$.
\end{example}

\begin{thm} \label{thm:vecdfs}
Associating to each $\mathbf{x}$-parking function $\alpha \in \mathrm{PF}(\mathbf{x})$ the triple $(T,\ell,\prec)$ produced by Algorithm~\ref{alg:vecdfs} on input $\alpha$ defines a bijection
\[ \psi \colon \mathrm{PF}(\mathbf{x}) \to \left\{ (T,\ell,\prec)\colon \parbox{3in}{\begin{center} $T \in \mathrm{RPT}(n+1), \; \prec \in \mathrm{AVO}(T), \; \ell\colon E(T) \to \mathbb{N},$ \\ $\ell(e) \in \{0,1,\ldots,\omega_{G_{\mathbf{x}}}(e)-1\} \textrm{ for all } e \in E(T)$ \end{center} } \right\}\]
This bijection satisfies $\mathrm{rsum}(\alpha) = \kappa(G_{\mathbf{x}},T,\prec) + \sum_{e \in E(T)}\ell(e)$ when $(T,\ell,\prec) = \psi(\alpha)$. The inverse of the bijection is given by Algorithm~\ref{alg:vecdfsinv}.
\end{thm}
\begin{proof}
Let $\alpha \in  \mathrm{PF}(\mathbf{x})$  and let $(T,\ell,\prec)$  be the output of the vector DFS-burning algorithm on input $\alpha$. First let us show that indeed $T \in \mathrm{RPT}(n+1)$ and $\prec \in \mathrm{AVO}(T)$. It is clear that~$T$ is a tree. We need to show that in fact $T$ has $n+1$ vertices. Suppose to the contrary that~$i := \#V(T)< n+1$. Let $G, \sigma$ be as in Algorithm~\ref{alg:vecdfs} when the algorithm terminates. Set $U := V(G) - \sigma(V(T))$. All edges in $G$ between $U$ and $\sigma(V(T))$ were burnt without any vertex in $U$ starting to burn, which means that for each $j \in U$ we have~$\alpha_j \geq \mathrm{deg}^{G}_U(j) = \sum_{j=0}^{i}x_j$. Thus there are at least $n-i$ entries of $\alpha$ greater than or equal to $\sum_{j=0}^{i}x_j$. But this means that $\alpha$ is not an $\mathbf{x}$-parking function, a contradiction. So indeed $V(T) = \{0,1,\ldots,n\}$. We have~$T \in \mathrm{RPT}(n+1)$ because the algorithm adds vertices to $T$ in a depth-first search fashion; the vertices of $T$ have been relabeled in such a way that vertex~$i$ was the $(i+1)$st vertex added to $T$. We have~$\prec \in \mathrm{AVO}(T)$ because the algorithm prefers to visit vertices~$j$ with a greater label first.

Continue to fix some~$\alpha \in  \mathrm{PF}(\mathbf{x})$ with $(T,\ell,\prec) := \psi(\alpha)$. Let $\sigma \in \mathfrak{S}_n$ be the unique permutation with $\sigma(1) \prec \sigma(2) \prec \cdots \prec \sigma(n)$. Set $\beta := (\alpha_{\sigma^{-1}(1)},\alpha_{\sigma^{-1}(2)},\ldots,\alpha_{\sigma^{-1}(n)})$. The crucial observation is the following: a run of the vector DFS-burning algorithm with input $\alpha$ simulates a run of the multigraph DFS-burning algorithm (Algorithm~\ref{alg:multidfs}) with input $\beta$ with respect to multigraph~$G_{\mathbf{x}}$ and vertex order~$\prec$. Thus, because the algorithm terminates with~$V(T) = V(G_{\mathbf{x}})$, Theorem~\ref{thm:multidfs} implies that $\beta \in \mathrm{PF}(G_{\mathbf{x}})$. Moreover, we have that~$(T,\ell) = \varphi^{\prec}(\beta)$. One consequence of $(T,\ell) = \varphi^{\prec}(\beta)$ is that~$\ell\colon E(T) \to \mathbb{N}$ satisfies~$\ell(e) \in \{0,1,\ldots,\omega_{G_{\mathbf{x}}}(e)-1\}$ for all~$e \in E(T)$. We conclude that~$\psi$ is well-defined. Another consequence of $(T,\ell) = \varphi^{\prec}(\beta)$, again appealing to Theorem~\ref{thm:multidfs}, is
\[\mathrm{rsum}(\alpha) = \mathrm{rsum}(\beta)= \kappa(G_{\mathbf{x}},T,\prec) + \sum_{e \in E(T)} \ell(e).\] 
Furthermore, Theorem~\ref{thm:multidfs} implies that $\psi$ is injective with inverse given by Algorithm~\ref{alg:vecdfsinv}.

To finish the proof we need to show that $\psi$ is surjective, or equivalently that when we apply the inverse algorithm to a valid triple $(T,\ell,\prec)$ we obtain an $\mathbf{x}$-parking function. So let~$T \in \mathrm{RPT}(n+1)$, $\prec \in \mathrm{AVO}(T)$, and~$\ell\colon E(T) \to \mathbb{N}$ with $\ell(e) \in \{0,1,\ldots,\omega_{G_{\mathbf{x}}}(e)-1\}$ for all $e \in E(T)$. Let $\alpha \in \mathbb{N}^{n}$ be the output of the inverse algorithm (Algorithm~\ref{alg:vecdfsinv}) on input $(T,\ell,\prec)$. Let $\beta$ be as in Algorithm~\ref{alg:vecdfsinv}. Because~$T$ has a depth-first search labeling, and because~$\prec$ is an admissible vertex ordering, the order that vertices of~$G_{\mathbf{x}}$ start to burn as we carry out the inverse algorithm is precisely~$0,1,2,\ldots,n$. Thus for any~$1 \leq i \leq n$, at most $\mathrm{deg}^{G_{\mathbf{x}}}_{\{i,i+1,\ldots,n\}}(i) - 1= \left( \sum_{j=0}^{i}x_j\right) -1$ chips were added to $i$ as we carried out the inverse algorithm. This means that~$\beta_i \leq  \left(\sum_{j=0}^{i}x_j\right) - 1$ for all~$1 \leq i \leq n$, which implies that $\beta \in \mathrm{PF}(\mathbf{x})$. But $\alpha$ is a permutation of $\beta$ and therefore~$\alpha \in \mathrm{PF}(\mathbf{x})$ as well.
\end{proof}

As mentioned, Theorem~\ref{thm:main} follows immediately from Theorem~\ref{thm:vecdfs}. Now let us see why Corollary~\ref{cor:main} follows from Theorem~\ref{thm:main}.

\begin{proof}[Proof of Corollary~\ref{cor:main}]
It suffices to show~$\#\mathrm{AVO}(T) = n! \prod_{i=1}^{n} (\mathrm{outdeg}_{T}(i-1)!)^{-1}$ for all~$T \in \mathrm{RPT}(n+1)$. To see this, let~$T \in \mathrm{RPT}(n+1)$. A total order~$\prec$ on~$\widetilde{V}(T)$ is admissible if and only if for each $i = 0,1,\ldots,n$ the children of $i$ appear in decreasing order of label in $\prec$. Clearly $(\mathrm{outdeg}_{T}(i)!)^{-1}$ of the $n!$ total orders on~$\widetilde{V}(T)$ have the children of $i$ in the appropriate order. Moreover, these conditions are all independent since each vertex has only a single parent. Lastly, note that~$n$ cannot have a child in~$T$ so we can omit the term $(\mathrm{outdeg}_{T}(n)!)^{-1}$ from the product.
\end{proof}

\begin{remark} \label{rem:corequivalent}
Let us show how Corollary~\ref{cor:main} is equivalent to Theorem~\ref{thm:pitmanstanley} in a simple way. The point is that $T \mapsto \gamma(T) := (\mathrm{outdeg}_T(0),\mathrm{outdeg}_T(1),\ldots,\mathrm{outdeg}_T(n-1))$ defines a bijection between $\mathrm{RPT}(n+1)$ and~$\Gamma(n)$. Why is this? Let $T \in \mathrm{RPT}(n+1)$. First of all, $\gamma(T)$ is indeed in~$\Gamma(n)$ since
\begin{itemize}
\item $\sum_{i=1}^{j} \gamma(T)_i \geq j$ for~$1 \leq j \leq n$  because~$\mathrm{par}_T(i) \in \{0,1,\ldots,j-1\}$ for~$1 \leq i \leq j$;
\item  $\sum_{i=1}^{n} \gamma(T)_i \leq n$ because only~$\{1,\ldots,n\}$ are children in $T$.
\end{itemize}
Moreover, it is not hard to see that for any~$i \in V(T)$ the set of descendants of $i$ in $T$ is precisely~$\{i+1,i+2,\ldots,i+m\}$ where 
\[m := \mathrm{max}\{k \in \{0,1,\ldots,n-i\}\colon \gamma(T)_{i+1} +  \gamma(T)_{i+2} + \cdots + \gamma(T)_{i+j} \geq j \textrm{ for all } 1 \leq j \leq k\}.\]
But knowledge of the descendants of $i$ for each~$i \in V(T)$ determines the tree~$T$. So the map $T \mapsto \gamma(T)$ is injective. That $T \mapsto \gamma(T)$ is bijective then follows from the well-known facts, mentioned in Section~\ref{sec:intro}, that $\#\mathrm{RPT}(n+1) = C_n = \#\Gamma(n)$.
\end{remark}

\section{Open questions and future directions}

The following are some possible threads of future research.

\begin{enumerate}[(1)]

\item It is often worthwhile to substitute $q:=1,0,-1$ in any statistical enumerator for a combinatorial set. Specializing $q := 1$ in the reversed sum enumerator for $\mathrm{PF}(\mathbf{x})$ of course yields  $\# \mathrm{PF}(\mathbf{x})$. Specializing $q:=0$ counts the number of maximal $\mathbf{x}$-parking functions: these are just permutations of the partition
\[\lambda^{\mathbf{x}} := (x_n+x_{n-1}+\cdots+x_1-1,\ldots,x_2+x_1-1,x_1-1)\] 
and so are easily enumerated. Chebikin and Postnikov~\cite{chebikin2010generalized} (see also Pak and Postnikov~\cite{pak1994resolvents}) showed that the absolute value of the specialization $q:=-1$ is $0$ if $\lambda^{\mathbf{x}}_1$ is even and is~$\beta_{n}(S_\mathbf{x})$ if $\lambda^{\mathbf{x}}_1$ is odd. Here~$\beta_n(S)$ for~$S \subseteq \{1,2,\ldots,n-1\}$ is the number of permtutations in $\mathfrak{S}_n$ with descent set~$S$, and~$S_{\mathbf{x}} := \{ i \colon \lambda^{\mathbf{x}}_{i+1} \textrm{ is odd},1\leq i \leq n-1\}$. Via Theorem~\ref{thm:main}, the~$q:=1$ specialization yields the expression for~$\# \mathrm{PF}(\mathbf{x})$ in terms of trees stated in Corollary~\ref{cor:main}. It would be interesting to interpret both the~$q:=0$ and~$q:=-1$ specializations in terms of trees via Theorem~\ref{thm:main}. The~$q:=0$ specialization should lead to an extension of the notion of \emph{increasing tree} (see e.g.~\cite[Proposition 1.5.5]{stanley2012ec1}). Meanwhile, one could hope to find a involution on rooted plane trees with admissible vertex orders that inverts $\kappa$-number parity and recovers, by consideration of fixed points, the aforementioned formula for~$q:=-1$.

\item A \emph{weakly increasing $\mathbf{x}$-parking function} is $\alpha = (\alpha_1,\ldots,\alpha_n) \in \mathrm{PF}(\mathbf{x})$ that satisfies~$\alpha_1\leq \alpha_2 \leq \cdots \leq \alpha_n$. Denote the set of weakly increasing $\mathbf{x}$-parking functions by $\mathrm{PF}^{\nearrow}(\mathbf{x})$. It is not hard to see that weakly increasing $\mathbf{x}$-parking functions are in bijection with partitions $\mu = (\mu_1,\ldots,\mu_n)$ with $\mu \subseteq \lambda^{\mathbf{x}}$ (here $\mu \subseteq \lambda^{\mathbf{x}}$ means that~$\mu_i \leq \lambda^{\mathbf{x}}_i$ for all~$1 \leq i \leq n$). Moreover, this bijection yields the identity
\[ \sum_{\alpha \in \mathrm{PF}^{\nearrow}(\mathbf{x})} q^{\mathrm{rsum}(\alpha)} = \sum_{\mu \subseteq \lambda^{\mathbf{x}}} q^{|\lambda^{\mathbf{x}}| - |\mu|}\]
where for a partition $\nu = (\nu_1,\ldots,\nu_n)$ the \emph{size} of $\nu$ is $|\nu| := \sum_{i=1}^{n} \nu_i$. It would be interesting to give an expression for the reversed sum enumerator for~$\mathrm{PF}^{\nearrow}(\mathbf{x})$ in terms of trees and inversions. The DFS-burning algorithm does not respect weak increasingness in any obvious way. Specializing $q:=1$ of course yields $\# \mathrm{PF}^{\nearrow}(\mathbf{x})$ or equivalently the number of partitions $\mu$ contained in a given partition $\lambda^{\mathbf{x}}$. This number is given by a determinant that was essentially known to MacMahon~\cite[p.~243]{macmahon1960combinatory}; see the discussion after Theorem 12 in~\cite{stanley2002polytope}.

\item Although Theorem~\ref{thm:gpfandxpf} classifies the graphs whose parking functions are also vector parking functions, there can still be exceptional equalities between the numbers of parking functions for other families of graphs and vectors. Here is an example of such an exceptional equality. Let $K^{a}_{m,m}$ denote the complete bipartite graph with edge-weights $a$; e.g.~$V(K^{a}_{m,m}) := \{0,1,\ldots,2m-1\}$ and
\[\omega_{K^{a}_{m,m}}(\{i,j\}) := \begin{cases} a &\textrm{if $j-i \equiv 1 \mod 2$} \\ 0 &\textrm{otherwise}.\end{cases}\] 
Set~$\mathbf{y}(a,2,m) := (a,0,a,0,\ldots,a) \in \mathbb{N}^{2m-1}$. Then for all $a \geq 1$, $m \geq 2$,
\[\#\mathrm{PF}(K^{a}_{m,m}) = a^{2m-1}m^{2m-2} = \#\mathrm{PF}(\mathbf{y}(a,2,m)).\] 
The map~$\alpha \mapsto 2\alpha$ bijects between~$\mathrm{PF}(\mathbf{y}(1,2,m))$ and the subset of $\mathrm{PF}(2m-1)$ for which all entries are even, so we might call $\mathrm{PF}(\mathbf{y}(1,2,m))$ the set of ``even parking functions.'' We can slightly extend this example by considering directed graphs. There is a notion of~$G$-parking function for $G$ a directed graph (see~\cite{postnikov2004trees, chebikin2005family}). Let~$G(a,b,m)$ be the directed multigraph with~$V(G(a,b,m)) := \{0,1,\ldots,bm-1\}$ and with directed edge-weight function
\[\omega_{G(a,b,m)}(i,j) := \begin{cases}a & \textrm{if $j-i \equiv 1 \mod b$} \\ 0  &\textrm{otherwise}.\end{cases}\] 
Note that~$G(a,2,m) = K^{a}_{m,m}$ and~$G(a,1,m) = K^{a,a}_m$ are actually undirected graphs. Set
\[\mathbf{y}(a,b,m) := (a,\overbrace{0,0,\ldots,0}^{b-1},a,\overbrace{0,0,\ldots,0}^{b-1},\ldots,a,\overbrace{0,0,\ldots,0}^{b-2}) \in \mathbb{N}^{bm-1}\] 
Then for all $a,b \geq 1$ and $m \geq 2$,
\[\#\mathrm{PF}(G(a,b,m)) = a^{bm-1}m^{bm-2} = \#\mathrm{PF}(\mathbf{y}(a,b,m)).\] 
Apparently $\mathrm{PF}(G(a,b,m)) \neq \mathrm{PF}(\mathbf{y})$ for all choices of~$a \geq 1$, $b \geq 2$, and $m \geq 3$. It is likely possible to extend the classification of graphs whose sets of parking functions are invariant under~$\mathfrak{S}_n$ in Theorem~\ref{thm:gpfandxpf} to include directed graphs but such a classification is beyond the scope of the present paper. At any rate, it would be very interesting if the equality~$\#\mathrm{PF}(G(a,b,m)) = \#\mathrm{PF}(\mathbf{y}(a,b,m))$ were more than just a numerical coincidence: is there some reasonable bijection between the two sets? Even in the case $a := 1$ and $b:=2$, we are not aware of a bijection between even parking functions and spanning trees of the complete bipartite graph.

\end{enumerate}

\bibliography{parking_functions_tree_inversions}{}
\bibliographystyle{plain}

\end{document}